\documentclass{article}
\usepackage{amsmath,amssymb}
\usepackage{enumerate,a4}
\usepackage{xcolor,colortbl}
\usepackage{hyperref,color}
\usepackage{graphicx, here}
\usepackage{multicol}
\usepackage[font=footnotesize,labelfont=bf]{caption}
\usepackage{booktabs}
\usepackage{multirow}
\usepackage{tabularx}
\usepackage{siunitx}
\usepackage{tikz}
\usepackage{amsthm}
\usepackage{adjustbox}
\usepackage{comment}
\usepackage{arydshln}
\usepackage{algorithm2e}
\usepackage{verbatim}
\usepackage{afterpage}

\newtheorem{theorem}{Theorem}
\newtheorem{proposition}[theorem]{Proposition}

\newtheorem{definition}[theorem]{Definition}
\newtheorem{example}[theorem]{Example}
\newtheorem{remark}[theorem]{Remark}
\newtheorem{observation}[theorem]{Observation}
\newtheorem{lemma}[theorem]{Lemma}
\newtheorem{corollary}[theorem]{Corollary}

\DeclareMathOperator{\best}{\textit{best}}

\title{Quantum computing and the stable set problem}
\author{
	{Aljaž Krpan} \thanks{Rudolfovo, Science and Technology Centre Novo mesto, Slovenia} \thanks{Faculty of mathematics and physics, University of Ljubljana, Slovenia, {\tt krpan.aljaz@gmail.com}}
	\and {Janez Povh}\thanks{Rudolfovo, Science and Technology Centre Novo mesto, Slovenia, {\tt janez.povh@rudolfovo.eu}} \thanks{Faculty of mechanical engineering, University of Ljubljana, Slovenia} \and  {Dunja Pucher}\thanks{Department of Mathematics,
		University of Klagenfurt, Austria,
		{\tt dunja.pucher@aau.at} }}

\date{}

\begin{document}
	
	\maketitle
	
	\begin{abstract}
		\noindent 
		Given an undirected graph, the stable set problem asks to determine the cardinality of the largest subset of pairwise non-adjacent vertices. This value is called the stability number of the graph, and its computation is an NP-hard problem. In this paper, we solve the stable set problem using the D-Wave quantum annealer. By formulating the problem as a quadratic unconstrained binary optimization problem with the penalty method, we show its optimal value equals the graph's stability number for specific penalty values. However, D-Wave's quantum annealer is a heuristic, so the solutions may be far from the optimum and may not represent stable sets. To address these, we introduce a post-processing procedure that identifies samples that could lead to improved solutions. Additionally, we propose a partitioning method to handle larger instances that cannot be embedded on D-Wave's quantum processing unit. Finally, we investigate how different penalty parameter values affect the solutions' quality. Extensive computational results show that the post-processing procedure significantly improves the solution quality, while the partitioning method successfully extends our approach to medium-size instances.
		\\[1mm]
		\noindent\textbf{Keywords:} The stable set problem, Quantum annealing, Graph partitioning, D-Wave\\[1mm]
		\noindent\textbf{Math. Subj. Class. (2020): Primary 90C27, Secondary 81P68}	
	\end{abstract}

	\section{Introduction}\label{intro}
	
	\subsection{Motivation} The stable set problem is a fundamental combinatorial optimization problem. Given an undirected simple graph $G$, the stable set problem asks for the largest subset of vertices that are pairwise non-adjacent, known as the stability number of the graph. It is well-known that stable sets in $G$ correspond to cliques in the complement graph $\overline{G}$, making the stable set problem in $G$ equivalent to the maximum clique problem in $\overline{G}$. The maximum clique problem (and thus the stable set problem), is an NP-hard problem. The decision variant of the maximum clique problem is included in Karp's 1972 list of NP-complete problems~\cite{Karp72}. Hence, unless P~=~NP, no polynomial time algorithm can exactly solve the stable set problem in general. However, there are numerous real-life applications of this problem, for instance, in bioinformatics and chemoinformatics~\cite{Dognin}, coding theory~\cite{Etzion}, and scheduling~\cite{Weide}. Therefore, efficient algorithms are needed to solve this problem. 

    The most desired algorithms to solve it are the exact ones. They guarantee to find the optimal solution to the problem. Since the stable set problem belongs to the family of discrete optimization problems, the exact algorithms usually employ a variant of complete enumeration, like branch-and-bound or branch-and-cut algorithms. There is a long list of exact solvers, which can be grouped as academic solvers, developed and maintained by some research group and available by some variant of open source license, like BiqBin  \cite{Gusmeroli,hrga2021biqbin}, SCIP \cite{BolusaniEtal2024OO}, and commercial solvers, developed and maintained by some company and available under different commercial licenses, like Gurobi \cite{gurobi} and CPLEX \cite{cplex}. Other solvers to solve the stable set problem can be found also on  GAMS webpage \url{https://www.gams.com/latest/docs/S_MAIN.html}. 
    Nevertheless, exact solvers have significant limitations: they are always computationally very expensive, in both theory and practice and can solve only small and medium-sized instances with the underlying graph that has up to approximately $1500$ vertices; see, for instance,~\cite{Depolli} and~\cite{Segundo}.
    
    Alternatives to exact solvers are the algorithms that find good, but not necessarily optimal, solutions, the so-called heuristics. They can process large problems within acceptable computational time and usually find good solutions that are acceptable for real-life applications. However, they have limitations, and the distance from the optimal solution is often unknown. In a review of classical algorithms for the stable set problem from 2015~\cite{Wu}, $21$ different heuristics were compared. It turned out that the general swap-based multiple neighborhood tabu search approach from~\cite{Yan} yielded the best results. A recent review on novel approaches~\cite{marino} points out the algorithm from~\cite{Lamm} as extremely effective. The respective algorithm combines local search with evolutionary and extended kernelization techniques.

    \subsection{Quantum annealing}\label{intro_QA}
    In recent years, a new heuristic approach for solving combinatorial optimization problems has emerged based on quantum annealing. Quantum annealers leverage quantum mechanics principles to tackle optimization problems. They operate by exploiting one of the quantum phenomena, quantum tunneling, to find the most efficient solution to complex computational challenges. A review of mathematical and theoretical foundations of quantum annealing can be found in~\cite{Morita}.

    State-of-the-art quantum annealing machines integrate advanced technologies to achieve better qubit coherence times, error rates, and connectivity between qubits. These machines rely on superconducting circuits cooled to ultra-low temperatures, harnessing quantum effects to explore vast solution spaces efficiently. 

    D-Wave Systems has been at the forefront of quantum annealing technology, pioneering the development of quantum annealers. One of the distinguishing features of their quantum annealers is the qubit architecture; see~\cite{boothby:2020}. While traditional qubits rely on a coherent superposition of states, D-Wave utilizes an approach where qubits represent a probabilistic combination of binary states. This allows for an adaptive search through solution spaces, exploiting quantum tunneling to navigate complex landscapes efficiently. Further details on the inner workings of D-Wave's quantum annealer can be found in their documentation~\cite{what-is-quantum-annealing-d-wave-docs}.

    However, challenges persist in achieving fault-tolerant quantum computation. Noise and maintaining quantum coherence remain significant hurdles. D-Wave's machines are subject to limitations due to factors like thermal fluctuations and imperfections in the qubit layout. As a result, achieving quantum advantage—where quantum computers outperform classical computers for specific tasks—remains an ongoing goal rather than a realized standard in practical applications.

    Despite these challenges, D-Wave's quantum annealers have demonstrated promising results in tackling optimization problems. Indeed, several solvers for mathematical optimization problems are an integral part of D-Wave systems. The most basic one, the quantum processing unit (QPU), is declared a purely quantum solver. It is capable of solving optimization problems of small size. 

    Besides QPU, D-Wave offers a hybrid solver that combines classical and quantum computing to solve optimization problems. Details about this solver are not disclosed but based on information provided in~\cite{mcgeoch2020d}, the hybrid solver combines classical and quantum computing to solve problems beyond the reach of QPU. This solver seems to use a classical computer to preprocess the problem and generate a set of subproblems that can be solved using QPU. When the subproblems are solved, the results are returned to the classical computer, which combines them to produce a solution to the original problem.

    Although there is increased interest in using D-Wave systems for combinatorial optimization problems and several authors have already published numerical results obtained by D-Wave systems, see Section~\ref{Related_work}, a thorough analysis of the results for the stable set problem, especially how close are the results computed by D-Wave systems to the global optima,  as well as the development of strategies to enhance these results, are still missing.

    \subsection{Our contribution} 

    In this paper, we extend our research presented in~\cite{KRPAN_2023, povh} and address the aforementioned limitations. We consider the stable set problem and investigate methods to obtain high-quality solutions with algorithms that at least partially exploit the D-Wave systems.

    The main contributions of this paper are:
    \begin{enumerate}
        \item We formulate the stable set problem as a quadratic unconstrained binary optimization (QUBO) problem using a penalty approach, where the quadratic constraints, weighted by the penalty parameter $\beta$, are added to the cost function. We show that if we use an exact QUBO solver, then for $\beta\ge \frac{1}{2}$, the optimum value of QUBO formulation equals the stability number of the graph.
        \item The D-Wave QPU solver could only compute local optimum solutions for instances with less than $150$ vertices. Our results show that very often, these solutions are quite far away from the global solutions in terms of the objective value and may not even be stable sets. To address these, we developed a post-processing procedure that enables us to detect samples that could lead to improved solutions and to extract solutions that are stable sets. Also, we provide theoretical guarantees about the quality of the extracted solutions. The procedure is outlined in Algorithm~\ref{alg}. Table \ref{table_1} contains detailed numerical results of the solution provided by the QPU solver as part of Algorithm~\ref{alg}. 
        \item For instances with more than $150$ vertices, we used the D-Wave hybrid solver and the classical simulated annealing solver from the D-Wave’s library for interacting with QPU. Both solvers again returned solutions that were often not stable sets, so we applied the post-processing procedure from Algorithm~\ref{alg}, and we reported detailed numerical results in Tables~\ref{table_2} and~\ref{table_3}. A comparison of the results from these tables reveals that the solutions obtained using the hybrid solver and simulated annealing are very similar. This may partially explain how the hybrid solver is running. 
        \item For certain instances with more than $150$ vertices, we also developed an alternative to the hybrid solver, which we call the simple partitioning procedure. It enhances the so-called CH-partitioning from \cite{GHGG, Djidjev2016}, which divides the original problem into many smaller subproblems, solves them with the QPU, and finally takes the largest stable set across the subproblems as the final solution. In order to improve solutions returned by the QPU, we combine the simple CH-partitioning with the post-processing procedure, as outlined in Algorithm~\ref{alg_part}. Numerical analysis is presented in Table~\ref{table_4}.
        \item  All solvers used in this paper are only approximate solvers. Therefore, the theoretical guarantee mentioned in the first item does not apply, so we performed extensive numerical tests using different values of the penalty term $\beta$ in the range from $\frac{1}{10}$ to $100$ to evaluate which values of $\beta$ are most relevant.
        We observed that the value $\beta =\frac{1}{2}$ gives the overall best solutions, which is consistent with theoretical guarantees. Nevertheless, in combination with the post-processing procedure, $\beta$ smaller than $\frac{1}{2}$ also often gives good, sometimes even optimal, solutions, but this holds only for smaller graphs and does not scale with the graph size. 
    \end{enumerate} 
  
    \subsection{Outline and terminology} 

    The rest of this paper is organized as follows. In Section~\ref{Related_work}, we give an overview of related work that aims to solve the stable set problem with quantum annealing. A QUBO formulation related to the stable set problem is derived in Section~\ref{section_QUBO_formulation}. In Section~\ref{Section_post_processing}, we discuss some aspects and shortcomings of the given QUBO formulation and present our post-processing procedure. For larger instances, we propose to use the CH-partitioning method and develop its modified version in Section~\ref{Graph_partitioning}. In Section~\ref{computational_results}, we perform extensive computational study and present results for newly proposed approaches. Finally, we conclude with a short discussion of our results in Section~\ref{Conclusion}.
    
    Throughout this paper, $G = (V(G), E(G))$ denotes a simple undirected graph with $\vert V(G) \vert = n$ vertices and $\vert E(G) \vert = m$ edges. Without loss of generality, we assume that the set of vertices $V(G)$ is $V(G) = \{v_1, \ldots, v_n\}$. If a graph $G$ is clear from the context, we write shortly $G = (V, E)$. For a graph $G = (V, E)$, a stable set $S \subseteq V$ is a subset of pairwise non-adjacent vertices. A clique $C \subseteq V$ is a subset of pairwise adjacent vertices. The cardinality of the maximum stable set in $G$ is denoted by $\alpha(G)$, and the cardinality of the maximum clique in $G$ is denoted by $\omega(G)$. The stable set problem asks to determine $\alpha(G)$, while the maximum clique problem asks to determine $\omega(G)$. 
    
    A non-empty graph $G$ is called connected if any of its two vertices are linked by a path in $G$. A maximal connected subgraph of $G$ is called a component of $G$. The complement of a graph $G = (V, E)$ is the graph $\overline{G} = (V, \overline{E})$, where $\overline{E}$ is the set of non-edges in $G$, that is $\overline{E} = \{ v_i,v_j \in V\; \vert \; v_i \neq v_j \wedge \{v_i,v_j\} \notin E\}$. Hence, from the definitions it follows $\alpha(G) = \omega(\overline{G})$ and $\alpha(\overline{G}) = \omega(G)$. The relationship between the stable set and the maximum clique problem is visualized in Figure~\ref{figure_1}.
    
    \begin{figure}[htbp]
    \centering
    \begin{minipage}{0.45\textwidth}
    \centering
    \begin{tikzpicture}
    \tikzset{Bullet/.style={fill=black,draw,color=#1,circle,minimum size=3pt,scale=1}}
    \node[Bullet=black,label=left :{$v_5$}] (5) at (0,1) {};
    \node[Bullet=black,label=above :{$v_1$}] (1) at (2, 3) {};
    \node[Bullet=lightgray,label=right :{$v_2$}] (2) at (4, 1) {};
    \node[Bullet=black,label=right :{$v_3$}] (3) at (3, -1) {};
    \node[Bullet=lightgray,label=left :{$v_4$}] (4) at (1, -1) {};
    \draw (1) -- (4);
    \draw (2) -- (3);
    \draw (2) -- (4);
    \draw (3) -- (4);
    \draw (2) -- (5);
    \end{tikzpicture}
    \end{minipage}%
    \hfill
    \begin{minipage}{0.45\textwidth}
    \centering
    \begin{tikzpicture}
    \tikzset{Bullet/.style={fill=black,draw,color=#1,circle,minimum size=3pt,scale=1}}
    \node[Bullet=black,label=left :{$v_5$}] (5) at (0,1) {};
    \node[Bullet=black,label=above :{$v_1$}] (1) at (2, 3) {};
    \node[Bullet=lightgray,label=right :{$v_2$}] (2) at (4, 1) {};
    \node[Bullet=black,label=right :{$v_3$}] (3) at (3, -1) {};
    \node[Bullet=lightgray,label=left :{$v_4$}] (4) at (1, -1) {};
    \draw (1) -- (2);
    \draw (1) -- (3);
     \draw (1) -- (5);
    \draw (3) -- (5);
    \draw (4) -- (5);
    \end{tikzpicture}
    \end{minipage}
    \caption{The left figure shows a graph $G$ on five vertices. In this graph, the following subsets of vertices are stable sets: $\{v_1, v_2\}$, $\{v_1, v_3\}$, $\{v_1, v_5\}$, $\{v_1, v_3, v_5\}$, $\{v_3, v_5\}$, $\{v_4, v_5\}$. The right figure shows the complement graph $\overline{G}$, where these same subsets form cliques. In particular, the maximum stable set $\{v_1, v_3, v_5\}$ in $G$ is equivalent to the maximum clique in $\overline{G}$, and therefore, $\alpha(G) = \omega(\overline{G}) = 3$.} 
    \label{figure_1}
    \end{figure}   

    For a vertex $v_i \in V$, the set of vertices adjacent to $v_i$ is called its neighborhood. The closed neighborhood of vertex $v_i$ includes the vertex $v_i$ and its neighborhood. We denote this closed neighborhood set by $N(v_i)$. For $X \subseteq V$, we denote a subgraph of $G$ induced by $X$ by $G[X]$. The closed neighborhood of vertices in $X \subseteq V$ is $N(X) = \cup_{v_i \in X} N(v_i)$. The vector of all-ones of length $n$ is denoted by $e_n$, and we write $I_n$ for the identity matrix of order $n$. If $n$ is clear from the context, we write shortly $e$ and $I$.

	\section{Related Work}\label{Related_work}
	
	This section summarizes previous research on solving the stable set and maximum clique problems using quantum annealers, explicitly focusing on results from D-Wave systems. First, we discuss the architecture, methods, instances considered, and results obtained. Then, we highlight one aspect of the problem formulation related to the penalization parameter $\beta$ in prior studies and briefly present our contribution within this framework.

    One of the first numerical results on solving the stable set problem with  D-Wave quantum annealers is shown in~\cite{Parekh}. The results were obtained on D-Wave Two AQO with $512$ qubits in Chimera graph architecture. The experiments were compared to Selby's exact and heuristic algorithms for Chimera graphs, revealing that Selby's exact method performed better than the heuristic and quantum annealing approaches, since it found optimal solutions in a shorter running time.
    
    Further numerical experiments were conducted in~\cite{GHGG}, where the authors compared quantum annealing implementation to several classical algorithms, such as simulated annealing, Gurobi, and some third-party clique-finding heuristics. For their tests, they used D-Wave 2X with roughly $1000$ qubits in Chimera graph architecture. The experiments were performed on random graphs with $45$ vertices and edge probabilities ranging from $0.3$ to $0.9$. D-Wave returned solutions of comparable quality to classical methods, but classical solvers were generally faster for small instances. They also tested subgraphs of D-Wave chimera graphs, where D-Wave returned the best solutions for large instances and showed a substantial computing speed-up. For instances that did not fit D-Wave's architecture, they proposed decomposing the problem into subproblems and combining solutions, which was effective for sparse graphs.

    Similar experiments, but on D-Wave 2000Q, were done in~\cite{YPB}. The results obtained for random graphs with up to $60$ vertices and edge probabilities ranging from $0.05$ to $0.425$ were compared to classical algorithms such as simulated annealing and the graphical networks package NetworkX. It was shown that the D-Wave's QPU can find optimal or near-optimal solutions for all considered instances. However, it was outperformed by simulated annealing, which always found the optimal solution and was always the fastest for problems with $30$ or more vertices. Additionally, the authors assessed the performance of D-Wave 2000Q as a function of edge probability. After running $100000$ sample solutions per instance, they calculated the mean probability of obtaining the optimal solution. For instances with density $0.20$, the lowest success rate was noted, suggesting that these instances represent the most difficult instances for the QPU in terms of computing the optimum solution. 

    The decomposition method for instances that do not fit D-Wave's architecture introduced in~\cite{GHGG} was further considered in~\cite{Pelofske:2019}. A pruning strategy for subproblems arising from the decomposition was introduced to reduce the computational complexity. Also, some reduction strategies were considered. The experiments for selected DIMACS instances were performed on D-Wave 2000Q, and it was shown that combining the decomposition method with proposed strategies is faster at computing maximum cliques than the method introduced in~\cite{GHGG}, with a considerable speed-up for very sparse as well as very dense graphs. Additionally, in~\cite{Pelofske:2022xho}, the algorithm from~\cite{Pelofske:2019} was combined with parallel quantum annealing and applied to graphs with up to $120$ vertices and $6935$ edges on the D-Wave Advantage System 4.1. in Pegasus graph architecture. It was found that this combined approach can compute maximum cliques in high-density graphs significantly faster than classical solvers.

    Finally, a computational study of three generations of D-Wave quantum annealers was done recently in~\cite{pelofske:2023}. The experiments were conducted on random graphs with $52$ vertices and edge probabilities ranging from $0.05$ to $0.95$. The computations were performed on systems D-Wave 2000Q, D-Wave Advantage System 4.1, D-Wave Advantage System 6.1 as well as on the prototype of D-Wave Advantage2 System 1.1. in Zephyr graph architecture. The best results for the maximum clique problem were obtained on the newest prototype of D-Wave Advantage 2, showing a progression of the quantum annealing technology over a relatively short time span of development.
    
    To solve an optimization problem with D-Wave's QPU, one should formulate it as an Ising model or a QUBO problem. In~\cite{Lucas}, respective formulations for hard problems were introduced. Among others, a QUBO formulation for the stable set was given. The given formulation contains two terms, one for the vertices and one for the edges. While the contribution of a vertex in the solution is rewarded, every edge in the solution is penalized. It is suggested in~\cite{Lucas} to set penalization for edge to be strictly greater than a reward for a vertex. 

    In the related work discussed in this section, the authors focused on penalizing edges without rewarding vertices. Specifically, when adapting their formulations to the QUBO formulation given in this work---where vertices are not rewarded, but edges are penalized using a penalty parameter $\beta$---we observe that a value of $\beta = 1$ was used in all these studies.

    In contrast, we demonstrate that, with an exact QUBO solver, the optimal value of the given QUBO formulation corresponds to the stability number of the graph when the penalty term $\beta$ is set to at least $\frac{1}{2}$. However, given the approximate nature of quantum annealing, the solutions obtained may not always be optimal or they may not be stable sets. To address these issues, we introduce methods for analyzing and post-processing computed solutions to improve them. Our computational analysis suggests that setting the penalization term at $\beta = \frac{1}{2}$ yields better outcomes than the choice of $\beta = 1$ used in previous research.

	\section{QUBO formulation}\label{section_QUBO_formulation}
	
    In this section, we use the formulation of the stable set problem as an integer linear programming problem in binary decision variables with many linear constraints, which is obviously not an instance of QUBO. We first find an equivalent formulation where all linear constraints are replaced by a single quadratic constraint, and in the next step, we apply a penalty approach and move the new quadratic constraint, multiplied by a penalty parameter $\beta$, into the objective function, resulting in a QUBO problem.

    The optimum solution of this QUBO problem is not necessarily the same as the optimum solution of the original stable set problem, which means that the optimum solution of QUBO may not be a stable set, and its size may not be equal to the size of the maximum stable set in the underlying graph.

    We investigate which values of $\beta$ ensure that the optimum solution of QUBO is also the optimum solution of the original problem and provide theoretical guarantees in Section~\ref{subsection_QUBO_formulation}. We also illustrate what can happen if they are not fulfilled. However, choosing the appropriate value for the penalty parameter is a challenge when working with quantum annealers due to the so-called scaling problem, which we briefly discuss in Section~\ref{scaling_problem}.
    
    \subsection{Formulation of the stable set problem}\label{subsection_QUBO_formulation}
    
    Let $G = (V, E)$ be a graph and $S\subseteq V$ a stable set. We can represent $S$ by a binary vector $x$ of length $\vert V \vert$ with $x_i=1$ if and only if $v_i\in S$. Since the vertices in $S$ cannot be adjacent, for every edge at most one end vertex can be in $S$. A traditional way to formulate the stable set problem is the following integer linear programming formulation:
    \begin{equation}\label{stable_set_quadratic}
    \begin{aligned}
    \alpha(G)~=~\max \{e^Tx \mid ~x \in \{0, 1\}^n \wedge ~\forall \{v_i, v_j\} \in E,\; x_i+x_j \le 1\}.
    \end{aligned}
    \end{equation}
    
    Linear constraints in~\eqref{stable_set_quadratic} can be replaced by quadratic ones. The condition that for every edge, at most one end vertex can be in $S$, means that if $v_i \in S$, then $ x_j = 0$ for all $\{v_i, v_j\} \in E$. Hence, the stability number of a graph $G$ can be computed as the optimal value of the following optimization problem with linear objective function and quadratic constraints:
    \begin{equation}\label{stable_set_quadratic_2}
    \begin{aligned}
    \alpha(G)~=~\max \{e^Tx \mid ~x \in \{0, 1\}^n \wedge ~\forall \{v_i, v_j\} \in E,\; x_ix_j = 0\}.
    \end{aligned}
    \end{equation}
    
    The problem~(\ref{stable_set_quadratic_2}) can be further reformulated. Since $x_ix_j = 0$ for every edge $\{v_i, v_j\} \in E$, we have that $\sum_{\{v_i, v_j\} \in E} x_ix_j = 0$. We can also show the reverse claim. Indeed, since $x$ is a binary vector, $\sum_{\{v_i, v_j\} \in E} x_ix_j = 0$ implies that for every edge $\{v_i, v_j\} \in E$ we have  $x_i x_j = 0$. Therefore, binary vector $x$ satisfies the edge constraints $x_ix_j = 0$ for all $\{v_i, v_j\} \in E$ if and only if it satisfies the single constraint $\sum_{\{v_i, v_j\} \in E} x_ix_j = 0$. Now, let $A$ be the adjacency matrix of a graph $G$. Then, $\sum_{\{v_i, v_j\} \in E} x_ix_j = \frac{1}{2}x^TAx$. Obviously, $\frac{1}{2}x^TAx = 0 \Leftrightarrow  x^TAx = 0$, so the formulation~(\ref{stable_set_quadratic_2}) is equivalent to the formulation:
    \begin{equation}\label{stable_set_quadratic_3}
    \begin{aligned}
    \alpha(G)~=~\max \{e^Tx \mid ~x \in \{0, 1\}^n \wedge ~x^TAx = 0\}.
    \end{aligned}
    \end{equation}

    We apply a penalty method to~\eqref{stable_set_quadratic_3} to obtain a QUBO problem, which will be the central problem of this paper. Generally, a QUBO problem is a minimization problem of the following form:
    \begin{equation*}
    \begin{aligned}
    z^* ~=~ \min \{f^Tx + x^TQx \mid x \in \{0, 1\}^n \}. 
    \end{aligned}
    \end{equation*}

    To get a QUBO formulation related to the stable set problem, we need to move the quadratic constraint $x^TAx = 0$ in \eqref{stable_set_quadratic_3} into the objective function. 
    The penalty method suggests simply subtracting the left-hand side of this constraint, multiplied by a penalty parameter $\beta > 0$, from the objective function. The new objective function would, therefore, be $e^Tx-\beta x^TAx$. Since the QUBO problem is a minimization problem, we need also to reverse the sign of the objective function.

    With all this, we obtain the following QUBO problem closely related to the stable set problem:
    \begin{equation}\label{QUBO_formulation}
    \begin{aligned}
     \alpha(G,\beta) ~=~ - \min \{-e^Tx+\beta x^TAx \mid x\in \{0,1\}^n \}.
    \end{aligned}
    \end{equation}

    For the sake of clarity, we denote the cost function in \eqref{QUBO_formulation} as:
    \begin{equation*}
    \begin{aligned}
    \text{cost}(x, \beta) = -e^Tx + \beta x^TAx.
    \end{aligned}
    \end{equation*}
    Furthermore, for a subset of vertices $X \subseteq V$, we define:
    \begin{equation}\label{formulation_QUBO}
    \begin{aligned}
    \text{cost}(X, \beta) = \text{cost}(x, \beta),
    \end{aligned}
    \end{equation}
    where $x \in \{0, 1\}^n$ is the indicator vector for $X$. 
    The following observation can help us to better understand the role of the penalty parameter $\beta$.

    \begin{observation}\label{observation_penalty_term}
    Suppose $X\subset V$, $v_i,v_j\in X$ and there is an edge between $v_i,v_j$. Therefore, we have $x_i = x_j = 1$, and the term $x^TAx$ counts this edge twice. Since \eqref{QUBO_formulation} is a minimization problem, we have that for $\beta > 0$, the penalty term $\beta x^TAx$ penalizes each edge in $G[X]$ with the value of $2 \beta$.
    \end{observation}
    
    Based on Observation~\ref{observation_penalty_term}, we can rewrite the cost function~\eqref{formulation_QUBO} as:
    \begin{equation}\label{cost_function_rewitten}
    \begin{aligned}
        \text{cost}(X,\beta) = -\vert X \vert + 2\beta \vert E(G[X]) \vert .
    \end{aligned}
    \end{equation}

    Note that the optimal value of \eqref{QUBO_formulation} may not be equal to $\alpha(G)$ for some values of the parameter $\beta$. Likewise, the optimal solution $X$ may not be a stable set. In the following, we provide insights into how $\alpha(G)$ and $\alpha(G,\beta)$ are related for different values of $\beta$.
    
    \begin{lemma}\label{minimal_beta}
    Let $G = (V, E)$ be a graph. If $\beta \geq \frac{1}{2}$, then 
    $$\alpha(G) = \alpha(G,\beta).$$ 
    \end{lemma}

    \begin{proof}
    Let $x^* \in \{0, 1\}^n$ be an optimal solution for the problem (\ref{stable_set_quadratic_3}). Since the formulation (\ref{stable_set_quadratic_3}) is equivalent to the formulation (\ref{stable_set_quadratic_2}), we have that $x^*_i x^*_j = 0$ for all $\{v_i, v_j\} \in E$. Hence, the optimal solution $x^*$ is a stable set. Furthermore, $x^*$ is a feasible solution for the formulation~(\ref{QUBO_formulation}) with $-e^Tx^*+\beta(x^*)^TAx^*=-e^Tx^*$, and therefore $\alpha(G) \leq \alpha(G, \beta)$.
    
    Now let $x^* \in \{0, 1\}^n$ be an optimum solution for \eqref{QUBO_formulation} and $X^*$ the corresponding subset of $V$. If $X^*$ is a stable set, then the optimum value of \eqref{QUBO_formulation} is $e^Tx^*$, hence  $\alpha(G)\ge \alpha(G,\beta)$. Otherwise, there exists $\{v_i, v_j\} \in E$, such that $x^*_i = x^*_j = 1$. If we define new $\tilde x$, which differs from $x^*$ in the $j$-th component only, which is set to $0$, this implies $e^T\tilde x=e^Tx^*-1$ and $\beta \tilde x^TA\tilde{x}$ $\leq$ $\beta  (x^*)^TAx^*-2\beta$. Therefore, this change has the following effect on the cost function of~\eqref{QUBO_formulation}:
    \begin{align*}
    \text{cost}(\tilde x, \beta) &= -  e^T\tilde{x}  + \beta\tilde x^TA\tilde{x} \\
    & \leq - e^Tx^* + 1 + \beta(x^*)^TAx^*-2\beta \\
    & = \text{cost}(x^*, \beta) + 1 - 2\beta \\
    & \leq \text{cost}(x^*, \beta)
    \end{align*}
    where the second inequality follows from the fact that $\beta \geq \frac{1}{2}$.
    
    Therefore, since by assumption we have $\beta \geq \frac{1}{2}$, we can iteratively remove one endpoint for each remaining edge between two vertices in $G[X^*]$ without increasing the value of the cost function, which finally yields a stable set $\hat X$ with indicator vector $\hat x$ such that: 
    $$\text{cost}(x^*, \beta) \geq \text{cost}(\hat x, \beta) = -e^T\hat x \ge -\alpha(G).$$
    Hence, $\alpha(G,\beta) = -\text{cost}(x^*, \beta) \le \alpha(G)$. Combining with the first part of the proof, we get the equality.
    \end{proof}

    If $\beta < \frac{1}{2}$, then the equality $\alpha(G) = \alpha(G,\beta)$ may not be true, as the following counterexample shows.
    
    \begin{example}
    Let $G = (V, E)$ with $V = \{v_1, v_2\}$, $E = \{\{v_1, v_2\}\}$ and assume $\beta < \frac{1}{2}$. For $x=(1, 1)^T$, we have that
    $\text{cost}(x, \beta) = -e^Tx + \beta x^TAx =-2+2\beta<-1$.
    Therefore, the minimum of $\text{cost}(x, \beta)$ over all binary vectors is smaller than $-1$, and thus $\alpha(G,\beta) > 1 = \alpha(G)$.
    \end{example}

    \begin{corollary}\label{beta_small}
    Let $G = (V, E)$ be a graph, and let $\beta < \frac{1}{2}$. Then we have \begin{equation*}
    \begin{aligned}
    \alpha(G, \beta) \geq \alpha(G).
    \end{aligned}
    \end{equation*}      
    \end{corollary}

    \begin{proof}
    The inequality $\alpha(G, \beta) \geq \alpha(G)$ follows from the first part of the proof of Lemma~\ref{minimal_beta}, which is independent of $\beta$.
    \end{proof}

     As the last step, we formulate the stable set problem so that it can be embedded on D-Wave QPUs. D-Wave solvers are designed to minimize a QUBO formulated as:
    \begin{equation*}\label{QUBO_dwave}
    \begin{aligned}
    \min \quad & x^TQx,
    \end{aligned}
    \end{equation*}
    where $x$ is a binary variable. Since $x_i^2 = x_i$, we can write $-e^Tx$ as $x^T(-I)x$. Thus, we set $Q = -I + \beta A$ and obtain the following problem formulation:
    \begin{equation}\label{QUBO_2}
    \begin{aligned}
    \alpha(G,\beta)=- \min \{ x^T(-I+\beta A)x \mid x \in \{0, 1\}^n \}. 
    \end{aligned}
    \end{equation} 

    As mentioned in Section~\ref{Related_work}, several QUBO formulations for the stable set problem or related to the stable set problem similar to the formulations~(\ref{QUBO_formulation}) and (\ref{QUBO_2}) have already been introduced. However, the value of the penalty parameter was usually set to $\beta = 1$. This is interesting since the choice of parameters may impact the quality of solutions obtained from quantum annealers, as discussed in the following text.

    \subsection{The scaling problem for quantum annealers}\label{scaling_problem}

    Several authors investigated the relationship between the values of the penalty terms in a QUBO formulation and the quality of results; see, for instance,~\cite{Choi} and~\cite{Roch}. Nevertheless, although it is empirically proven that such a relationship exists, there is still no exact rule on how to set penalty terms efficiently. One reason is the scaling problem that occurs when working with QPUs. For the sake of brevity, we will briefly present this problem by giving a toy example and omit the theoretical framework. An interested reader can find further details in~\cite{Yarkoni}.

    To present the scaling problem, we give an example similar to the one outlined in D-Wave's documentation~\cite{auto-scaling-chain-strength-d-wave-docs}. We consider the following function, given as a QUBO problem:
    $$
    E_1(q_1, q_2) = 0.1q_1 + 0.2q_2 + 0.3q_1q_2.
    $$
    This function can also be written in a form:
    $$
    E_2(q_1, q_2) = 0.1q_1 + 0.2q_2 + 5q_1q_2,
    $$
    which is equivalent to $E_1$ in the sense that the values of the cost function for the solutions $E_1$ and $E_2$ follow the same order, as it can be seen from the following table:
    \begin{center}
    \begin{tabular}{c c|c|c}
    $q_1$ & $q_2$ & $E_1$ & $E_2$\\
    \hline
    0 & 0 & 0 & 0\\
    0 & 1 & 0.2 & 0.2\\
    1 & 0 & 0.1 & 0.1\\
    1 & 1 & 0.6 & 5.3\\
    \end{tabular}
    \end{center}

    However, a quantum annealer does not operate on the whole interval of the given coefficients, but it scales the coefficients to its operating space. Assume a quantum annealer operates on the coefficient interval $[-1, 1]$. Since the coefficients of the function $E_2$ exceed this interval, the quantum annealer will correct them by multiplying the function by $\frac{1}{5}$. This yields the new corrected function $E_2'$: 
    $$
    E_2'(q_1, q_2) = 0.02q_1 + 0.04q_2 + q_1q_2,
    $$
    with the following table of states:
    \begin{center}
    \begin{tabular}{c c|c}
    $q_1$ & $q_2$ & $E_2'$\\
    \hline
    0 & 0 & 0\\
    0 & 1 & 0.04\\
    1 & 0 & 0.02\\
    1 & 1 & 1.06\\
    \end{tabular}
    \end{center}
    As we can see, the difference between the lowest and the second-lowest solution has decreased from $0.1$ in $E_2$ to $0.02$ in $E_2'$. This is a problem because a quantum annealer is an analog device sensitive to noise. Thus, the smaller the difference between the states of the cost function, the greater the probability of error due to noise---this is called the scaling problem. Therefore, the absolute relative scale between coefficients should be as small as possible.

    In the given QUBO~\eqref{QUBO_2}, we only have one penalty parameter---the parameter $\beta$. Since we are interested in the relative scaling between the constants, we set $\beta$ in such a manner that the absolute values of the coefficients of all monomials are as similar as possible. Since the monomials $x_i^2$ have the coefficient $-1$, we note that $2\beta$ should be equal or approximately equal to $\vert-1\vert$, as this mitigates the scaling problem the most. This implies that $\beta$ should be close to $\frac{1}{2}$.

	\section{Analysis of solutions and the post-processing procedure}\label{Section_post_processing}
	
	In this section, we provide an in-depth analysis of data returned by a quantum annealer and develop methods that enable us to improve the quality of the results. In Section~\ref{Analysis_solution}, we first focus on the solutions returned by a quantum annealer and study how to set the penalty parameter $\beta$ for the most accurate results. Then, we analyze all samples returned by a quantum annealer in Section~\ref{Post_processing}. Finally, we propose the post-processing procedure and give its detailed explanation in Section~\ref{subsection_post_processing}.

    \subsection{Analysis of solution}\label{Analysis_solution}

    In Section~\ref{section_QUBO_formulation}, we derived two equivalent QUBO formulations \eqref{QUBO_formulation} and \eqref{QUBO_2} for the stable set problem and showed that the optimal solution yields the stability number of a graph if the value of the penalty parameter $\beta$ is at least $\frac{1}{2}$. However, quantum annealing is a heuristic, so the solution obtained is not necessarily optimal. Moreover, since any binary vector is feasible for \eqref{QUBO_formulation} and \eqref{QUBO_2}, the solution provided by a quantum annealer is not necessarily a stable set, which means that the penalty term in the solution provided by a heuristic based on a QUBO formulation may not be zero.

    We now analyze how to set the penalty parameter $\beta$ for the most accurate results. First, we show how to estimate the size of a stable set contained in a solution to the QUBO problem that is not a stable set and how to extract such a solution. Furthermore, we argue that, within our method, the best choice for the value of the penalty parameter is $\beta = \frac{1}{2}$. This result confirms our previous observation that in the context of the scaling problem presented in Section~\ref{scaling_problem}, the parameter $\beta$ should be close to $\frac{1}{2}$.

    Let $X \subseteq V$ be a solution returned by a quantum annealer, and let $x \in \{0, 1\}^n$ be the indicator vector of $X$. If $\beta x^TAx\neq 0$, then the solution is not a stable set. Nevertheless, we can still obtain a stable set from $X$, as shown in the next statement.

    \begin{proposition}\label{pr:extract_stable_set}
    Let $G = (V, E)$ be a graph and let $X\subseteq V$. Then there exists a stable set $X'\subseteq X$, such that
    \begin{equation*}
    \begin{aligned}
    \vert X' \vert \geq \vert X \vert - \vert E(G[X]) \vert.
    \end{aligned}
    \end{equation*} 
    The computation of $X'$ is done in $\mathcal{O}( \vert E(G[X]) \vert + \vert V(G[X]) \vert)$ time.
    \end{proposition}
    
    \begin{proof}
    We consider the graph $G[X]$. If $\vert E(G[X]) \vert \geq \vert X \vert$, the statement trivially holds since we can take $X'$ to be a single vertex. Now assume that $\vert E(G[X]) \vert < \vert X \vert$. Then, iterating over the set $E(G[X])$, we remove for each edge one of its adjacent vertices. Since there are $\vert E(G[X]) \vert$ edges, we remove at most that many vertices. As a result, we obtain a stable set of size at least $\vert X \vert - \vert E(G[X]) \vert$.

    The calculation of the induced subgraph $G[X]$ takes $\mathcal{O} (|E(G[X])| + |V(G[X])|)$ time while iterating over all edges and removing adjacent vertices takes $\mathcal{O}(\vert E(G[X]) \vert)$ time. Hence, the computation of $X'$ is done in $\mathcal{O}(\vert E(G[X]) \vert + \vert V(G[X]) \vert)$ time.
    \end{proof}

    Proposition~\ref{pr:extract_stable_set} asserts that even if a solution of \eqref{QUBO_2} is not a stable set and the number of edges is sufficiently small, we can still efficiently extract a stable set. We will now argue that when we are solving \eqref{QUBO_2} with a heuristic, a very reasonable choice for the penalization parameter is still $\beta = \frac{1}{2}$. We start by giving the following example. 

    \begin{example}\label{example}
    Let $\beta = 1$. Then, the penalty component of the QUBO formulation \eqref{QUBO_2} is $x^TAx$. Expanding $x^TAx$, we notice that each quadratic term appears twice according to Observation~\ref{observation_penalty_term}, resulting in a penalty of $2$ for each edge. Let us assume there exists a stable set of size at least of $25$ included in a subset of vertices $X \subseteq V$ such that $\vert X \vert = 30$ and $\vert E(G[X]) \vert = 5$. Then, $\text{cost}(X, 1) = -20$, so stable sets of sizes $21, 22, 23, 24$ with respective costs of $-21, -22, -23, -24$ will be more favorable than the set $X$. Therefore, even if only approximate, the solver will likely not return the set $X$ as the optimum. On the other hand, as indicated in Proposition~\ref{pr:extract_stable_set}, we can easily extract a stable set of size $25$ from the set $X$, meaning that the cost function with $\beta=1$ is not accurately representing the desired solution quality. 
    \end{example}

    From the given example, we note that in the context of the previously given method, choosing the penalization parameter $\beta > \frac{1}{2}$ may discard superior solutions due to the imprecise cost function. To explain this in a more general context, we first need to prove the following statement.
    \begin{lemma}
    \label{le:epsilon_graph}
        For every $\epsilon > 0$ there exist a graph $G$ and a subset $X \subseteq V$ such that 
        \begin{align}\label{epsilon}
        2\epsilon\vert E(G[X]) \vert > 1.
        \end{align}
    \end{lemma}
    \begin{proof} We will prove the lemma by providing an example of such graph. Let $\varepsilon>0$ and set $m=\lceil\frac{1}{2\epsilon} + 1\rceil$. The graph which will serve as an example is a union of $m$ edges, i.e, we have  $\vert E \vert  = m$, $\vert V \vert = 2m$ and $\deg(v) = 1$ for every $v \in V$. For $X = V$ it follows that:
        $$
        \vert E(G[X]) \vert = \vert E(G[V]) \vert = |E| = m > \frac{1}{2\epsilon},
        $$
        so the inequality \eqref{epsilon} holds for this graph.
    \end{proof}
    
    Now, suppose we have a graph $G=(V,E)$, a set $X \subseteq V$ and $\beta =\frac{1}{2} + \epsilon$ where $\epsilon > 0$. The value of the cost function~\eqref{cost_function_rewitten} is:
    $$
    -\vert X \vert+2\beta \vert E(G[X]) \vert = -(\vert X \vert-(1+2\epsilon) \vert E(G[X]) \vert).
    $$
    To improve clarity and simplify the explanation, we will flip the sign of the cost function and approach the problem as a maximization problem. This adjustment transforms the cost function into:
    $$
    cost = \vert X \vert-(1+2\epsilon) \vert E(G[X])\vert.
    $$
    As established in Proposition~\ref{pr:extract_stable_set}, we can easily extract a stable set of size:
    $$
    solution = \vert X \vert- \vert E(G[X]) \vert
    $$
    out of the $X$. Since we are dealing with an arbitrary graph, we can select $G$ and the set $X \subseteq V$ such that:
    \begin{align*}
    solution - cost = 2\epsilon\vert E(G[X]) \vert > 1
    \end{align*}
    holds, which implies that $solution - 1 > cost$. Existence of such graph is ensured by Lemma~\ref{le:epsilon_graph}.
    Since the cost function value of a stable set is equal to the cardinality of the stable set itself, this means that some stable set $X'$ of cardinality $\vert X' \vert = solution-1$ has better cost function value than our set $X$, even though we can easily extract a stable set of size $solution$ out of the set $X$. Therefore, the cost function imprecisely evaluates our solution and can potentially favor some inferior solutions. To avoid this issue, $\epsilon$ should be set to $0$, ensuring that $solution - cost$ cannot exceed $1$. Consequently, $\beta = \frac{1}{2}$ is the most suitable parameter value. If it is any larger, the best candidate set might not be detected as we have just shown, and if it is smaller, Equation~\eqref{QUBO_formulation} does not hold anymore due to Corollary~\ref{beta_small}.

    \subsection{Analysis of samples}\label{Post_processing}

    We now analyze all other samples returned by a quantum annealer. When working with quantum annealers from D-Wave, we can specify how often QPU should run a problem once the problem has been programmed onto the QPU hardware. Each run of a problem is also known as a read, an annealing cycle, or a sample. A sample with the lowest cost function will be returned as the solution to the problem. In the following, we show that, when dealing with the stable set problem, the solution returned by the solver does not necessarily imply that other samples cannot potentially offer better results. Furthermore, we will develop a theory for identifying potential samples that could lead to improved outcomes and present the idea of our post-processing procedure.

    Let $X\subseteq V$ be a solution returned by a quantum annealer. If $X$ is not a stable set, we know from Proposition~\ref{pr:extract_stable_set} that we can extract a stable set of size $\vert X \vert- \vert E(G[X]) \vert$. Now assume that $X\subseteq V$ is any sample returned by a quantum annealer. Then, if $X$ is not a stable set, we could again employ the procedure from Proposition~\ref{pr:extract_stable_set} and extract a stable set. Nevertheless, we show that for samples that are not stable sets, we can obtain even better estimations of the size of the stable sets that can be extracted. 
    
    Let us return to Example~\ref{example}. For a subset of vertices $X \subseteq V$ with $\vert X \vert = 30$ and $\vert E(G[X]) \vert = 5$, Proposition \ref{pr:extract_stable_set} assures us that we can always extract a stable set of size at least $25$ from it. However, it may occur, for instance, that one of the vertices has a degree of $5$, and by removing it, we obtain a stable set of size $29$. This property is generalized with the annihilation number of the graph, a graph parameter that is also an upper bound on the stability number of a graph, as shown in~\cite{Pepper}.
    
    \begin{definition}\label{definition_an}
        For a graph $G = (V, E)$, where $d_1 \leq d_2 \leq \ldots \leq d_n$ for $d_i = d(v_i)$, the \textit{annihilation number} $a = a(G)$ is defined as the largest index $a$ such that $\sum_{i=1}^a d_i \leq m$, where $m$ is the number of edges.
    \end{definition}
    We have the following   relationship between the annihilation number and the stability number of a general graph.
     \begin{theorem}\label{th:upper_bound_pepper}(Pepper,~\cite{Pepper})
        For any graph $G$, 
        \begin{equation}\label{annihilation}
        \alpha(G) \leq a(G). 
        \end{equation}
    \end{theorem}

    The original paper details the justification for this upper bound, but the concept is rather straightforward. Let degrees of vertices be defined as above. In each iteration, we remove one vertex, and we repeat this until we have a stable set. To accelerate the formation of a stable set, we aim to remove the maximum number of edges in each iteration. Consequently, we prioritize removing the vertex with the highest degree in the graph, specifically vertex $v_n$. Removing the vertex of degree $d_n$, the sum of all degrees drops by $2d_n$. Here we assume the best case scenario, which is that this reduction occurs in vertices $v_i, \ldots, v_{i+d_n-1}$, where $v_i$ represents the first non-zero degree vertex. By preserving the highest degree vertices for as long as possible, we accelerate the attainment of a stable set. When we remove enough vertices, the total number of edges will drop to zero, leaving the remaining vertices as the upper bound for the stability number of $G$.

    The bound~(\ref{annihilation}) may not be effective for general graphs. However, we will not apply it to compute upper bounds for general graphs. Instead, we will use it to compute bounds for graphs induced from the samples returned by a quantum annealer. These induced graphs possess the property of being highly sparse. This is due to the annealer's objective of finding a set $X$ that minimizes the cost function $-\vert X \vert + 2\beta\vert E(G[X]) \vert$. Consequently, in most cases, even when $\beta = \frac{1}{2}$, the returned samples will have negative cost since otherwise, even a single vertex, constituting a stable set, would have a lower cost function than the set $X$. Working with sparse graphs implies low vertex degrees, so the aforementioned upper bound~\eqref{annihilation}, which considers vertex degrees, effectively reduces the number of samples that need to be evaluated.

    The statements of Proposition~\ref{pr:extract_stable_set} and Theorem~\ref{th:upper_bound_pepper} allow us to obtain bounds for the samples that are not stable sets and to identify whether these samples may improve the solution returned by a quantum annealer. To assess potential improvements, we need to recompute the stability numbers of such samples. Given that these samples are sparse, we can reduce the computational costs by considering their connected components, as stated in the next observation.

    \begin{observation}\label{connected_components}
        Let $G$ be a graph and let $C_1, C_2, \ldots, C_l$ be connected components of $G$. Then 
        $$\alpha(G) = \sum_{i=1}^{l} \alpha(C_i).$$
    \end{observation}

    Altogether, combining Proposition~\ref{pr:extract_stable_set}, Theorem~\ref{th:upper_bound_pepper}, and Observation~\ref{connected_components}, we develop the method for enhancing the solutions returned by a quantum annealer. In the following text, we provide a detailed explanation of this method, which we call the post-processing procedure.

    \subsection{Post-processing procedure}\label{subsection_post_processing}

    The main goal of the post-processing procedure is to improve the solution returned by a quantum annealer. For this purpose, we collect samples that could lead to better solutions and recompute their stability numbers. Nevertheless, this recalculation can be done using heuristic or exact methods, not necessarily a quantum annealer. It is also possible to apply a different penalty term when recalculating the stability number compared to the one used for the computations with a quantum annealer. To distinguish between the solvers utilized, we give the following definition.

    \begin{definition}
    We define solver $S(G, \beta)$ as a function which takes as input a graph $G$ and the value of the parameter $\beta$ and returns the array of solutions $[X] = [X_1, \ldots, X_k]$ such that $\text{cost}(X_1, \beta) \leq \text{cost}(X_2, \beta) \leq \ldots \leq \text{cost}(X_k, \beta)$. Furthermore, we define the computed value $\hat{\alpha}(G, \beta, S)$ as:
    $$\hat{\alpha}(G, \beta, S) = -\text{cost}(X_1, \beta).$$
   \end{definition}

   Let $S = S(G, \beta)$ denote the solver we use to obtain the initial array of solutions $[X] = [X_1, \ldots, X_k]$, and let $S_{post} = S(G[X_i], \beta_{post})$ denote the solver for recomputing the stability number of a sample $X_i$. The post-processing steps are as follows: first, we take a look at $X_1$, which has the lowest cost, and define the current best solution as $\best = \vert X_1 \vert - \vert E(G[X_1]) \vert$. If $X_1$ is a stable set, then $\vert E(G[X_1]) \vert = 0$, and $\vert X_1 \vert = -\text{cost}(X_1, \beta)$. Otherwise, from Proposition~\ref{pr:extract_stable_set}, we know that we can extract a stable set of size at least $\vert X_1 \vert - \vert E(G[X_1]) \vert$ from $X_1$.

   We go through each sample $X_i$ for $i=1, \ldots, k$ and check whether $a(G[X_i]) > \best$. If $G[X_i]$ is a stable set, then $a(G[X_i]) \leq \best$ trivially holds since samples are sorted by cost, implying that reevaluating the stability number for this sample cannot yield an improved solution. However, if $G[X_i]$ is not a stable set but $a(G[X_i]) > \best$, we recompute the solution on the induced graph $G[X_i]$ using $S_{post}$. For this purpose, we use Observation~\ref{connected_components} and calculate the stability number of the sample $G[X_i]$ by considering its connected components. If the new solution surpasses $\best$, we update its value. However, if the condition $a(G[X_i]) > best$ is not fulfilled, and since $a(G[X_i])$ is the upper bound for $\alpha(G[X_i])$ according to Theorem~\ref{th:upper_bound_pepper}, we know that the sample $X_i$ cannot offer a better solution than the one we currently have, so we can skip the recalculation on $X_i$.  
   
   In the end, the procedure returns the cardinality of the best solution found. We denote this value by $\hat{\alpha}_{post}([X], \beta_{post}, S_{post})$. The steps of the general post-processing procedure are outlined in Algorithm~\ref{alg}. 

    \SetKwComment{Comment}{/* }{ */}
    \RestyleAlgo{ruled}
    \begin{algorithm}
    \caption{Post-processing procedure, $\hat{\alpha}_{post}([X], \beta_{post}, S_{post})$}\label{alg}
    \KwData{$[X]$ from $S$, $\beta_{post}$, $S_{post}$} 
    $best \gets \vert X_1 \vert - \vert E(G[X_1])\vert $ \\    
    \For{$i \gets 1$ to $k$}{
    \If{$a(G[X_i]) > best$}
    {
      $[C_1, \ldots, C_l] \gets connected\_components(G[X_i])$\\
      $s \gets \sum_{j=1}^{l}\hat{\alpha}(C_j, \,\beta_{post},\, S_{post})$\\
      \If{$s > best$}{
      $best \gets s$
      }
    }
    }
    \textbf{return} $best$
    \end{algorithm}   

    Finally, we explain how we used Algorithm~\ref{alg} for our computations. As solver $S$, we mostly use D-Wave's quantum annealer. Nevertheless, we do not use QPU to post-process samples. This is because these samples are small compared to the overall problem and can be recomputed relatively easily and cheaply. We have chosen to recompute them with the simulated annealing algorithm, which is part of the D-Wave library we use to interact with QPU. Although the simulated annealing again returns an array of solutions, we just focus on the best solution and check whether it is a stable set. If the solution is a stable set, then we take the computed value. Otherwise, we extract a stable set by using the procedure presented in the proof of Proposition~\ref{pr:extract_stable_set} which removes vertices with edges one by one until a stable set is obtained.

    Regarding the parameter $\beta$, we will computationally study how different values of this parameter influence the quality of the solutions. If we make computations with $\beta \geq \frac{1}{2}$, then we set $\beta_{post} = \beta$. However, if we use $\beta < \frac{1}{2}$, we set $\beta_{post} = \frac{1}{2}$. We will perform computations with $\beta < \frac{1}{2}$ to explore the behavior of the QPU and assess whether it can produce useful samples for post-processing. In this scenario, accurate solutions from the QPU are not expected. Thus, to improve solution accuracy through post-processing, and as argued in Section~\ref{Analysis_solution}, we set $\beta_{post} = \frac{1}{2}$. Computational results for Algorithm~\ref{alg} are presented in Section~\ref{computational_results}.

	\section{Dealing with large instances}\label{Graph_partitioning}
	
	As mentioned in Section~\ref{intro_QA}, some graphs are too large to be directly embedded on the QPU. In such cases, we can employ the method of CH-partitioning and divide a graph into multiple smaller subgraphs, which can then be embedded on the QPU. This method considers a related problem, namely the maximum clique problem, but, as we will observe, the conversion between the maximum clique and the maximum stable set is trivial. Therefore, we can use the given method within our framework for finding the maximum stable set. In Section~\ref{CH_method}, we describe the CH-partitioning method, and in Section~\ref{simple_partitioning_text}, we outline a modified version of the CH-partitioning we will use in our work. Furthermore, in Section~\ref{Bounding}, we discuss bounding procedures that enable us to reduce the number of partitions we have to consider and present an algorithmic framework that combines the proposed modified partitioning with the post-processing procedure. 

    \subsection{CH-partitioning}\label{CH_method}

    The concept of CH-partitioning was initially introduced in~\cite{Djidjev2016} and then further considered in~\cite{GHGG} and is defined as follows.

    \begin{definition}\label{CH-definition}
    Given an input graph $G =(V, E)$ with $n$ vertices and a positive integer $s \leq n$, a CH-partition $\mathcal{P}$ of graph $G$ is defined as:
    \begin{align*}
    \mathcal{P} = \{(C_i, H_i)\;|\; 1 \leq i \leq s\},
    \end{align*}
    where
    \begin{enumerate}
        \item the sets $\emptyset \neq C_i \subset V$ are called core sets and the set $\{C_i\;|\;1\le i\le s\}$  is called \textit{core partitioning} since the core sets must partition $V$, i.e.:  $\bigcup_i C_i = V$ and $C_i \cap C_j = \emptyset$ for all $i\neq j$,
         \item there is one \textit{complementary} set $H_i$ of vertices, also refereed as \textit{halo} set, for each core set $C_i$, defined as $H_i := N(C_i)\setminus C_i$.
    \end{enumerate}
    The cost of a CH-partitioning $\mathcal{P}$ is defined as:
    \begin{align*}
    \mbox{cost}(\mathcal{P}) = \max_{1\leq i \leq s} (|C_i|+|H_i|).
    \end{align*}   
    The \textit{CH-partitioning problem} is finding a CH-partitioning of $G$ with minimum cost for some fixed $s$.
    \end{definition}

    The following proposition, which establishes a relation between the maximum clique problem and CH-partitioning, was shown in~\cite{GHGG}.

    {\begin{proposition}
        \label{old_ch_partitioning_and_cliques}
        Given a CH-partitioning $\mathcal{P}$ of a graph $G$, the size of the maximum clique of $G$ is equal to $\max_i\{k_i\}$, where $k_i$ is the size of a maximum clique of the subgraph of $G$, induced by $C_i \cup H_i$.
    \end{proposition}

    \begin{proof}
    Let $K$ be a maximum clique of $G = (V, E)$ and let $v$ be any vertex of $K$. Since, by definition of CH-partitioning, $\bigcup_i C_i = V$, there exists exactly one $j$ such that $v$ belongs to core set $C_j$, hence $v \in C_j \cup H_j$. To show the statement, it is sufficient to show that all other vertices of $K$ are contained in $C_j \cup H_j$.

    Let $w \in K$, $w \neq v$, and assume $w \in C_j$. Then, trivially, $w \in C_j \cup H_j$, and the statement holds.
    
    Now assume that $w \notin C_j$. Since $K$ is a clique, there is an edge between any two vertices from it, and hence there is an edge between $v$ and $w$ and $w \in N(v)$. Since, by definition, $H_j$ consists of all neighbors of vertices from $C_j$ that are not in $C_j$, and since $v \in C_j$ and $w \notin C_j$ according to assumption, it follows that $w \in N(C_j)\setminus C_j$. Hence, $w \in H_j$, and therefore $w \in C_j \cup H_j$, so the statement holds.
    \end{proof}

    Assume we use the presented CH-partitioning method in the following way. Let $G = (V, E)$ be a graph, where the vertices $V$ are ordered in some ordering, and set $C_i = \{v_i\}$ for all $i \in \{1, \ldots, n\}$. Then $H_i = N(v_i) \setminus \{v_i\}$ for all $i \in \{1, \ldots, n\}$. Altogether, this CH-partition consists of $n$ pairs $(C_i,H_i)$ and its cost is  $\Delta(G) + 1$. This is the CH-partition with the lowest cost function if we enumerate over all $s$ since each partition has to contain in each set $C_i$ at least one vertex and all of the neighbors of $C_i$. 

    An advantage of this approach is that the computation of the CH-partitioning itself is not computationally intensive. However, this method also has a disadvantage. Let us assume that $K$ is a maximum clique such that $\vert K \vert = k$. Then, $k$ partitions contain the same maximum clique. Thus, we perform computations on $k$ sets containing the same maximum clique. To overcome this obstacle, we use a modified version of the CH-partitioning. 

    \subsection{Simple CH-partitioning}\label{simple_partitioning_text}

    Our goal is to develop a procedure that reduces the number of partitions containing a maximum clique. To achieve this, we build upon the concept introduced in~\cite{Szabo2}. In that work, the authors explored upper bounds on clique numbers by disregarding subgraphs. We have now adapted this concept to our CH-partitioning framework. With this adaptation, we define the \emph{simple CH-partitioning} as follows.

    \begin{definition}\label{simple-CH-definition}
    Given an input graph $G =(V, E)$, a simple CH-partition $\mathcal{P}_S$ of graph $G$ is defined as:
    \begin{align*}
    \mathcal{P}_S = \{(C_i, H_i)\;|\; 1 \leq i \leq n\},
    \end{align*}
    where
    \begin{enumerate}
        \item the \textit{core partitioning} contains only the subsets with one element, i.e., the set of vertices $V$ is partitioned into $n$ nonempty disjoint \textit{core} sets, $C_i = \{ v_i\}$, for all $i \in \{1, \ldots, n\}$,
        \item the \textit{complementary} set $H_i$ of vertices for each core set $\{v_i\}$ is defined as $H_i := N(v_i) \setminus \{v_1, \ldots, v_i\}$.
    \end{enumerate}
    The cost of a simple CH-partitioning $\mathcal{P}_S$ is defined in a similar way  as before:
    \begin{align*}
    \mbox{cost}(\mathcal{P}_S) = \max_{1\leq i \leq n} (|C_i|+|H_i|).
    \end{align*}   
    \end{definition}

    An advantage of simple CH-partitioning is that the size of partitions will become smaller as the partition number approaches $n$, enabling us to disregard partitions too small to contain a maximum clique, thus reducing the computational expense. Furthermore, an appropriate vertex ordering could reduce the size of the largest partitions to less than $\Delta + 1$. Finally, we show that the statement of Proposition~\ref{old_ch_partitioning_and_cliques} can also be applied to this new method.

    \begin{proposition}
        \label{ch_partitioning_and_cliques}
        Given a simple CH-partitioning $\mathcal{P}_S = \{(C_i, H_i)\;|\; 1 \leq i \leq n\}$ of a graph $G$, the size of the maximum clique of $G$ is equal to $\max_i\{k_i\}$, where $k_i$ is the size of a maximum clique of the subgraph of $G$, induced by $C_i \cup H_i$ .
    \end{proposition}
    
    \begin{proof}
    Let $K$ be a maximum clique of $G = (V, E)$, and let $\vert K \vert = k$. Without loss of generality, let $K = \{v_{p_1}, v_{p_2}, \ldots, v_{p_k}\}$ such that $1 \leq p_1 < p_2 < \ldots < p_k \leq n$. For the sake of simplicity, we introduce also $L=\{v_1,v_2,\ldots, v_{p_1}\}$. Obviously, $K\cap L=\{v_{p_1}\}$.
    According to Definition~\ref{simple-CH-definition}, every vertex $v_j$ of $K$ is equal to the core set $C_j$. Selecting the core set containing the vertex $v_{p_1}$, we have that $H_{p_1} = N(v_{p_1}) \setminus L$. 
    For any $w\in K,~w\neq v_{p_1}$, we have  $w \in N(v_{p_1})\setminus L=H_{p_1}$, therefore, $K \subseteq C_{p_1} \cup H_{p_1}$.
    \end{proof}

    \begin{remark}
       \label{property_simple_CH}
        Let $\mathcal{P}_S$ be a simple CH-partitioning of a graph $G$, and let $K$ be a maximum clique in $G$. Then there exists exactly one $j \in \{1, \ldots, n\}$ such that $K \subseteq C_j \cup H_j$.
    \end{remark}

    Recall that $X \subseteq V$ is a (maximum) clique in $\overline{G}$ if and only if $X$ is a (maximum) stable set in $G$. Since $G = \overline{\overline{G}}$, based on the statement of Proposition~\ref{ch_partitioning_and_cliques}, we proceed as follows. Suppose we want to find a maximum stable set in graph $G$. We first define a simple CH-partition for graph $\overline{G}$, in which we sort the vertices according to predefined ordering. Then we use Proposition~\ref{ch_partitioning_and_cliques}, and obtain:
    $$\max_i\{\omega(\overline{G}[C_i \cup H_i])\} = \omega(\overline{G}) = \alpha(G).$$

    By definition of the graph complement, we know that $X_i$ is a maximum clique in $\overline{G}[C_i \cup H_i]$ if and only if $X_i$ is a maximum stable set in 
    $\overline{\overline{G}[C_i \cup H_i]}$.
    Combining all this, we get:
    $$\max_i\{\alpha(\overline{\overline{G}[C_i \cup H_i]})\} = \max_i\{\omega(\overline{G}[C_i \cup H_i])\} = \omega(\overline{G}) = \alpha(G).$$ 
    It is easy to see that $\overline{\overline{G}[X]}$ is isomorphic to $G[X]$, and therefore:
    $$
    \max_i\{\alpha(G[C_i \cup H_i])\} = \alpha (G).
    $$
    This means we can compute the maximum stable set in $G$ by computing maximum stable sets in subgraphs $G[C_i \cup H_i]$ and taking the largest one.

    Finally, we note that using simple partitioning greatly reduces the cost of partitioning. We consider some instances from the literature that cannot be directly embedded on QPU and compare the costs of regular CH-partitioning for $s = n$ with the costs for the proposed modified version, where we sort the vertices based on their degrees in increasing order. The results are presented in Table~\ref{table:simple_regular_part_comparison}.

    \begin{table}[hbt!]
    \caption{Comparison of costs of the simple and regular CH-partitioning for $s = n$}
    \begin{center}
    \begin{adjustbox}{max width=\textwidth}
    \begin{tabular}{ l r r r | r r | r r}
    \hline
    Instance & $n$ & $m$ & $d$ & $\mbox{cost}(\mathcal{P})$  &  $\mbox{cost}(\mathcal{P}_S)$  & Difference & Reduction of cost\\
    \hline
    brock200\_1     & 200 & 5066 & 0.25 & 166 & 136 & 30 & 18\% \\

    brock200\_2     & 200 & 10024 & 0.50 &  115 & 87 & 28 & 24\% \\

    brock200\_3     & 200 & 7852  & 0.39 & 135 & 109 & 26 & 19\% \\
 
    brock200\_4     & 200 & 6811  & 0.34 & 148 &120 & 28 & 19\% \\

    keller4         & 171 & 5100  & 0.35 & 125 & 103 & 22 & 18\% \\

    p\_hat500\_1    & 500 & 93181 & 0.75 & 205 & 95 & 110 & 54\% \\

    san200\_0\_7\_1 & 200 & 5970  & 0.30 & 156 & 131 & 25 & 16\% \\

    san200\_0\_7\_2 & 200 & 5970  & 0.30 & 165 & 123 & 42 & 25\% \\

    sanr200\_0\_7   & 200 & 6032  & 0.30 & 162 & 127 & 35 & 22\% \\

    c-fat200-1 & 200 & 18366 & 0.92 & 18 & 17 & 1 & 6\%\\
    c-fat200-2 & 200 & 16665 & 0.84 & 35 & 33 & 2 & 6\%\\
    c-fat200-5 & 200 & 11427 & 0.57 & 87 & 84 & 3 & 3\%\\
    c-fat500-1 & 500 & 120291 & 0.96 & 21 & 20 & 1 & 5\%\\
    c-fat500-2 & 500 & 115611 & 0.93 & 39 & 38 & 1 & 3\%\\
    c-fat500-5 & 500 & 101559 & 0.81 & 96 & 93 & 3 & 3\%\\
    \hline
    \end{tabular}
    \end{adjustbox}
    \end{center}
    \label{table:simple_regular_part_comparison}
    \end{table}
 
    \subsection{Bounding the partitions}\label{Bounding}

    The idea of pruning the subproblems in the context of CH-partitioning was previously explored in~\cite{Pelofske:2019}. For this purpose, the authors bounded partitions with several approaches, for instance, by coloring the partitions or computing the Lovász theta function. In this work, we use properties of the simple CH-partitioning and present two simple approaches.

    As the first step, we use the previously mentioned advantage of the simple CH-partitioning. Recall that the simple CH-partitioning yields altogether $n$ pairs $(C_i,H_i)$. Nevertheless, the size of partitions $(C_i,H_i)$ becomes smaller as $i$ approaches $n$. Now, assume that we have already computed or just estimated the largest stability number for some partition(s). Then, if the upper bound of some other partition is at most the best-computed value, the respective partition cannot contain a larger stable set. 
    
    Furthermore, we again use the annihilation number, as defined in~Definition \ref{definition_an}. Recall that the annihilation number $a(G)$ of a graph $G$ is the largest integer $k$ such that there exist $k$ different vertices in $G$ with the degree sum at most $m$, and represents an upper bound on the stability number of $G$, see~Theorem~\ref{th:upper_bound_pepper}. Thus, the annihilation number can be used as a reduction strategy: if the annihilation number of a partition is less than or equal to the size of the largest stable set found in other partitions, that partition can be disregarded. We also apply post-processing to every solution, removing excess edges and potentially further improving the solution.

    Following the given argumentation, we outline the steps of the simple partitioning procedure combined with the post-processing in Algorithm~\ref{alg_part}. 
    
    \begin{algorithm}
    \caption{Simple CH-partitioning with post-processing 
    $\hat{\alpha}_{post}^{part}(G, \beta, \beta_{post}, S, S_{post})$}\label{alg_part}
    \KwData{Graph $G$, $\beta$, $\beta_{post}$, $S$, $S_{post}$}
    $best\gets 0$\\
    Define vertex ordering $v_1, v_2, \ldots, v_n$\\
    \For{$i \gets 1$ to $n$}{
    $C_i \gets \{v_i\}$\\
    $H_i \gets N_{\overline{G}}(v_i)\setminus \cup_{k=1}^{i}\{v_k\}$\\
    \If{$a(G[C_i \cup H_i] )> best$}{
    $[X] \gets S(G[C_i \cup H_i],\beta)$\\
    $s \gets \hat{\alpha}_{post}([X], \beta_{post}, S_{post})$\\
    \If{$s > best$}   
    {
      $best \gets s$\\
    }
    }
    }
    \textbf{return} $best$
    \end{algorithm}

    In our computational implementation of Algorithm~\ref{alg_part}, we use solvers $S$, $S_{post}$, and penalty terms $\beta$, $\beta_{post}$ as described in Section~\ref{subsection_post_processing}, and we sort the vertices based on their degrees in increasing order.

    However, it is important to note that alternative vertex orderings could also be employed. Specifically, while introducing the concept of estimating clique size by discarding subgraphs in~\cite{Szabo2}, the authors explored different vertex orderings and found that these can lead to significantly different upper bounds. Therefore, investigating the impact of vertex ordering within the context of simple CH-partitioning combined with post-processing remains a topic for future research.

	\section{Computational results}\label{computational_results}
		
	We now computationally explore ideas presented in Sections~\ref{Section_post_processing} and~\ref{Graph_partitioning} and report computational results. We start by giving information about considered instances in Section~\ref{benchmarks}. In Section~\ref{Different_betas}, we study the effects of imposing different values of the penalization parameter $\beta$ in combination with the proposed post-processing procedure outlined in Algorithm~\ref{alg}. Finally, in Section~\ref{Larger_instances_partitioning_results}, we employ the simple CH-partitioning method combined with the post-processing procedure as presented in Algorithm~\ref{alg_part}.}  
    
    All computations presented in this work were conducted on an Intel(R) Core(TM) i7-9750H with 16GB of RAM. Additionally, we used two cloud solvers from the D-Wave Leap platform: the QPU solver \verb|Advantage_system4.1| and the hybrid solver \verb|hybrid_binary_quadratic_model_version2|. We also utilized a classical, locally run simulated annealing solver, which is part of the D-Wave's Python library we employed for interacting with the QPU. Finally, we note that the computations presented in this paper were conducted at the end of 2023 and the beginning of 2024. For these computations, the only hyperparameter we adjusted in the D-Wave solvers was the number of runs; all other hyperparameters were left at their default settings.

    \subsection{Benchmark set}\label{benchmarks}

    We consider several classes of instances from the literature: DIMACS instances, Paley graphs, and several evil instances.
    \begin{itemize}
    \item DIMACS instances: We perform experiments on instances from the Second DIMACS Implementation Challenge~\cite{Johnson1996Cliques}, which took place in 1992 and 1993 and which dealt with NP-hard problems maximum clique, graph coloring, and satisfiability. We consider benchmarks for the maximum clique problem. Additionally, we perform experiments on some of the \verb|dsjc| graphs. These graphs are contained in the part of the challenge dealing with graph coloring but are also interesting from the perspective of the maximum clique problem. Since we study the stable set problem, we always consider complements of the mentioned instances.
    \item Paley graphs: The maximum clique problem in Paley graphs has been considered by several authors; see, for instance,~\cite {Cohen} and~\cite{HansonPetridis}. We perform experiments on some small Paley graphs that can be embedded on QPU. Since Paley graphs are self-complementary, their clique numbers equal their stability numbers. Therefore, we do not complement these instances.
    \item Evil instances: Extremely hard and versatile (evil) instances for the maximum clique problem were recently introduced in~\cite{Szabo}. As noted in~\cite{marino}, these instances stand out for their ability to provide a wide range of difficulties, making them a valuable resource in algorithmic research and development. The benchmark is composed of altogether $40$ graphs. In this work, we consider the complements of graphs with at most $125$ vertices.
    \end{itemize}

    \subsection{Results for the post-processing procedure}\label{Different_betas}

    In this section, we assess how the post-processing procedure presented in Section~\ref{subsection_post_processing} improves the initial results from D-Wave's solvers and explore the relationship between result quality and different values of the parameter $\beta$.

    For this purpose, we perform two sets of experiments. In the first set of experiments, we consider $18$ instances with up to $125$ vertices that can be embedded on QPU. In the second set of experiments, we deal with $7$ large instances with up to $1500$ vertices that cannot be embedded on QPU and use the hybrid solver. In this case, we obtain only one solution from the solver. Hence, to fully analyze the potential of the post-processing procedure, we solve these instances with D-Wave's simulated annealing solver and post-process all returned samples.

    We conduct computations as described in Section~\ref{subsection_post_processing} and outlined in Algorithm~\ref{alg}. For the results presented in Section~\ref{results_small_instances}, we use QPU as solver $S$, while for the results presented in Section~\ref{result_large_instances}, we use hybrid solver as well as simulated annealing as initial solvers $S$. Additionally, we note that for the QPU and simulated annealing computations, we run $1000$ iterations. Simulated annealing is employed also as $S_{post}$ with $100$ iterations. We vary the parameter $\beta$, including lowering it to $\frac{1}{10}$. Note that this contradicts the assumption in Lemma~\ref{minimal_beta} requiring $\beta \geq \frac{1}{2}$, but we make this adjustment for exploratory reasons, aiming to understand the behavior of the QPU and evaluate its capability to generate valuable samples for post-processing. 

    \subsubsection{Results for smaller instances}\label{results_small_instances}
    
    The results for the first set of experiments are shown in Table~\ref{table_1}. Left of the vertical line, we provide general information about the considered instances: name of the graph, number of vertices $n$, number of edges $m$, edge density $d$, and the stability number of the graph $\alpha(G)$. Right of the vertical line, we provide the raw results and results with post-processing. The first two lines provide the raw results returned by QPU solver. The $\hat{\alpha}$ represents the highest value of the QUBO function, and the ``\# vertices, edges`` represents the number of vertices and the number of edges in the solution, which yielded this highest value. The last two lines provide the results we got using post-processing of raw results. The $\hat{\alpha}_{post}$ again represents the value of the QUBO function, which is computed on the solution we got back from post-processing. Since post-processing always returns a feasible solution, this number always represents the cardinality of the largest stable set found. The ``\# recalculations`` denotes the number of samples for which we have recomputed the results. If the obtained solution is a stable set and the solution value corresponds to $\alpha(G)$, we write the value in bold.
    
    \begin{table}[hbt!]
    \caption{Numerical results for the post-processing procedure for QUBO formulation with different values of penalty term $\beta$, where the QPU solver is used as the initial solver $S$ in Algorithm~\ref{alg} }
    \scriptsize
    \begin{center}
    \begin{adjustbox}{max width=\textwidth}
    \begin{tabular}{ l l r | l r r r r r r r r }
    \hline
    Instance & Data & Value & Result & $\beta = 1/10$ & $\beta = 1/8$ & $\beta = 1/6$ & $\beta = 1/4$ & $\beta = 1/2$ & $\beta = 1$ & $\beta = 10$ & $\beta = 100$ \\
    \hline
    C125.9 & $n$ & 125 & $\hat{\alpha}$ & 46.8 & 43.2 & 40.6 & 37 & 31 & 29 & 24 & 21 \\
    & $m$ & 787 & \# vertices, edges & 62, 76 & 57, 55 & 49, 25 & 41, 8 & 35, 4 & 29, 0 & 24, 0 & 21, 0\\ 
    & $d$ & 0.10 & $\hat{\alpha}_{post}$ &\textbf{34} & \textbf{34} & \textbf{34} & \textbf{34} & 32 & 29 & 24 & 22 \\
    & $\alpha(G)$ & 34 & \# recalculations & 1000 & 1000 & 973 & 155 & 1 & 0 & 0 & 1 \\
    \hline
    dsjc125.5 & $n$ & 125 & $\hat{\alpha}$ & 6.8 & 9 & 6.3 & 4.5 & 5 & 4 & 6 & 4 \\
    & $m$ & 3859 & \# vertices, edges & 14, 36 & 16, 28 & 11, 14 & 10, 11 & 5, 0 & 6, 1 & 6, 0 & 4, 0\\ 
    & $d$ & 0.50 & $\hat{\alpha}_{post}$ & 7 & 9 & 7 & 7 & 7 & 7 & 6 & 6\\
    & $\alpha(G)$ & 10 & \# recalculations & 1000 & 692 & 495 & 696 & 280 & 5 & 0 & 37 \\
    \hline
    dsjc125.9 & $n$ &  125 & $\hat{\alpha}$ & 45.2 & 41.2 & 39 & 35 & 32 & 26 & 23 & 21 \\
    & $m$ & 789 & \# vertices, edges & 62, 84 & 55, 55 & 46, 21 & 40, 10 & 34, 2 & 28, 1 & 23, 0 & 21, 0 \\ 
    & $d$ & 0.10 & $\hat{\alpha}_{post}$ & \textbf{34} & \textbf{34} & \textbf{34} & 33 & 32 & 27 & 23 & 22 \\
    & $\alpha(G)$ & 34 & \# recalculations & 1000 & 993 & 552 & 80 & 0 & 0 & 0 & 1 \\
    \hline
    evil\_chv12x10 & $n$ & 120 & $\hat{\alpha}$ & 45.8 & 40.8 & 35.3 & 30 & 20 & \textbf{20} & \textbf{20} & \textbf{20} \\
    & $m$ & 545 & \# vertices, edges & 69, 116 & 62, 85 & 51, 47 & 41, 22 & 28, 8 & 20, 0 & 20, 0 & 20, 0 \\  
    & $d$ & 0.08 & $\hat{\alpha}_{post}$ & \textbf{20} & \textbf{20} & \textbf{20} & \textbf{20} & \textbf{20} & \textbf{20} & \textbf{20} & \textbf{20}\\
    & $\alpha(G)$ & 20 & \# recalculations & 1000 & 999 & 1000 & 984 & 105 & 1 & 0 & 0 \\
    \hline 
    evil\_myc5x24 & $n$ & 120 & $\hat{\alpha}$ & 75.8 & 70 & 62.3 & 55 & 48 & \textbf{48} & 44 & 43 \\
    & $m$ & 236 & \# vertices, edges & 100, 121 & 97, 108 & 80, 53 & 67, 24 & 51, 3 & 48, 0 & 44, 0 & 43, 0 \\ 
    & $d$ & 0.03 & $\hat{\alpha}_{post}$ & \textbf{48} & \textbf{48} & 48 & \textbf{48} & \textbf{48} & \textbf{48} & 44 & 43 \\
    & $\alpha(G)$ & 48 & \# recalculations & 748 & 260 & 1000 & 698 & 1 & 0 & 0 & 0 \\
    \hline
    evil\_myc11x11 & $n$ & 121 & $\hat{\alpha}$ & 50 & 43.8 & 36.7 & 33 & 22 & \textbf{22} & \textbf{22} & \textbf{22} \\
    & $m$ & 508 & \# vertices, edges & 75, 130 & 6, 89 & 57, 61 & 44, 22 & 32, 10 & 22, 0 & 22, 0 & 22, 0 \\  
    & $d$ & 0.07 & $\hat{\alpha}_{post}$ & \textbf{22} & \textbf{22} & \textbf{22} & \textbf{22} & \textbf{22} & \textbf{22} & \textbf{22} & \textbf{22} \\
    & $\alpha(G)$ & 22 & \# recalculations & 1000 & 1000 & 1000 & 958 & 109 & 2 & 0 & 0 \\
    \hline
    evil\_s3m25x5 & $n$ & 125 & $\hat{\alpha}$ & 38.6 & 35.8 & 31.3 & 25 & 20 & 19 & 19 & 17 \\
    & $m$ & 873 & \# vertices, edges & 53, 72 & 49, 53 & 45, 41 & 36, 22 & 23, 3 & 19, 0 & 19, 0 & 17, 0 \\  
    & $d$ & 0.11 & $\hat{\alpha}_{post}$ & \textbf{20} & \textbf{20} & \textbf{20} & \textbf{20} & \textbf{20} & \textbf{20} & 19 & 18 \\
    & $\alpha(G)$ & 20 & \# recalculations & 1000 & 999 & 1000 & 997 & 139 & 2 & 0 & 1 \\
    \hline
    hamming6\_2 & $n$ & 64 & $\hat{\alpha}$ & \textbf{32} & \textbf{32} & \textbf{32} & \textbf{32} & \textbf{32} & \textbf{32} & \textbf{32} & \textbf{32} \\
    & $m$ & 192 & \# vertices, edges & 32, 0 & 32, 0 & 32, 0 & 32, 0 & 32, 0 & 32, 0 & 32, 0 & 32, 0 \\  
    & $d$ & 0.10 & $\hat{\alpha}_{post}$ & \textbf{32} & \textbf{32} & \textbf{32} & \textbf{32} & \textbf{32} & \textbf{32} & \textbf{32} & \textbf{32} \\
    & $\alpha(G)$ & 32 & \# recalculations & 0 & 0 & 0 & 0 & 0 & 0 & 0 & 0 \\
    \hline
    hamming6\_4 & $n$ & 64 & $\hat{\alpha}$ & 6.6 & 5.5 & 4 & 4 & \textbf{4} & 2 & \textbf{4} & \textbf{4} \\
    & $m$ & 1312 & \# vertices, edges & 12, 27 & 11, 22 & 8, 12 & 7, 6 & 4, 0 & 4, 1 & 4, 0 & 4, 0 \\ 
    & $d$ & 0.65 & $\hat{\alpha}_{post}$ & \textbf{4} & \textbf{4} & \textbf{4} & \textbf{4} & \textbf{4} & \textbf{4} & \textbf{4} & \textbf{4} \\
    & $\alpha(G)$ & 4 & \# recalculations & 994 & 989 & 923 & 78 & 88 & 74 & 11 & 43 \\
    \hline
    johnson8\_2\_4 & $n$ & 28 & $\hat{\alpha}$ & 7.2 & 6 & 5.3 & 4 & \textbf{4} & \textbf{4} & \textbf{4} & \textbf{4} \\
    & $m$ & 168 & \# vertices, edges & 12, 24 & 10, 16 & 8, 8 & 7, 6 & 4, 0 & 4, 0 & 4, 0 & 4, 0 \\ 
    & $d$ & 0.44 & $\hat{\alpha}_{post}$ & \textbf{4} & \textbf{4} & \textbf{4} & \textbf{4} & \textbf{4} & \textbf{4} & \textbf{4} & \textbf{4} \\
    & $\alpha(G)$ & 4 & \# recalculations & 955 & 998 & 560 & 51 & 0 & 0 & 0 & 0 \\
    \hline
    johnson8\_4\_4 & $n$ & 70 & $\hat{\alpha}$ & 16.8 & 14 & \textbf{14} & \textbf{14} & \textbf{14} & \textbf{14} & 12 & 11 \\
    & $m$ & 560 & \# vertices, edges & 22, 56 & 19, 20 & 14, 0 & 14, 0 & 14, 0 & 14, 0 & 12, 0 & 12, 0 \\  
    & $d$ & 0.23 & $\hat{\alpha}_{post}$ & \textbf{14} & \textbf{14} & \textbf{14} & \textbf{14} & \textbf{14} & \textbf{14} & 12 & 11 \\
    & $\alpha(G)$ & 14 & \# recalculations & 379 & 7 & 0 & 0 & 0 & 0 & 0 & 0 \\
    \hline
    johnson16\_2\_4 & $n$ & 120 & $\hat{\alpha}$ & 14.2 & 12 & 10.33 & 8 & 7 & 7 & 7 & \textbf{8} \\
    & $m$ & 1680 & \# vertices, edges & 25, 54 & 20, 32 & 17, 20 & 14, 12 & 9, 2 & 7, 0 & 7, 0 & 8, 0 \\ 
    & $d$ & 0.24 & $\hat{\alpha}_{post}$ & \textbf{8} & \textbf{8} & \textbf{8} & \textbf{8} & \textbf{8} & \textbf{8} & 7 & \textbf{8} \\
    & $\alpha(G)$ & 8 & \# recalculations & 1000 & 1000 & 989 & 558 & 15 & 2 & 2 & 0 \\
    \hline
    MANNa\_9 & $n$ & 45 & $\hat{\alpha}$ & 30.6 & 27 & 24 & 20 & 16 & \textbf{16} & \textbf{16} & \textbf{16} \\
    & $m$ & 72 & \# vertices, edges & 45, 72 & 39, 48 & 36, 36 & 27, 14 & 17, 1 & 16, 0 & 16, 0 & 16, 0 \\ 
    & $d$ & 0.07 & $\hat{\alpha}_{post}$ & \textbf{16} & \textbf{16} & 14 & 15 & \textbf{16} & \textbf{16} & \textbf{16} & \textbf{16} \\
    & $\alpha(G)$ & 16 & \# recalculations & 12 & 311 & 58 & 454 & 0 & 0 & 0 & 0 \\
    \hline
    paley61 & $n$ & 61 & $\hat{\alpha}$ & 8.4 & 7.5 & 6.7 & 6 & 5 & \textbf{5} & \textbf{5} & \textbf{5}  \\
    & $m$ & 915  & \# vertices, edges & 14, 28 & 12, 18 & 9, 7 & 8, 4 & 6, 1 & 5, 0 & 5, 0 & 5, 0 \\  
    & $d$ & 0.50 & $\hat{\alpha}_{post}$ & \textbf{5} & \textbf{5} & \textbf{5} & \textbf{5} & \textbf{5} & \textbf{5} & \textbf{5} & \textbf{5} \\
     & $\alpha(G)$ & 5 & \# recalculations & 1000 & 991 & 950 & 361 & 230 & 39 & 5 & 15\\
    \hline
    paley73 & $n$ & 73 & $\hat{\alpha}$ & 8.2 & 8 & 6.3 & 6 & 5 & 4 & \textbf{5} & \textbf{5} \\
    & $m$ & 1314 & \# vertices, edges & 13, 24 & 18, 16 & 11, 14 & 9, 6 & 6, 1 & 6, 1 & 5, 0 & 5, 0 \\  
    & $d$ & 0.50 & $\hat{\alpha}_{post}$ & \textbf{5}& \textbf{5} & \textbf{5} & \textbf{5} & \textbf{5} & \textbf{5} & \textbf{5} & \textbf{5} \\
    & $\alpha(G)$ & 5 & \# recalculations & 1000 & 1000 & 1000 & 724 & 151 & 152 & 55 & 58 \\
    \hline
    paley89 & $n$ & 89 & $\hat{\alpha}$ & 7.4 & 5 & 4.3 & 4.5 & 3 & 0 & \textbf{5} & 0  \\
    & $m$ & 1958 & \# vertices, edges & 13, 28 & 14, 36 & 11, 20 & 11, 13 & 7, 4 & 0, 0 & 5, 0 & 0, 0  \\ 
    & $d$ & 0.50 & $\hat{\alpha}_{post}$ & \textbf{5} & \textbf{5} & \textbf{5} & \textbf{5} & \textbf{5} & \textbf{5} & \textbf{5} & \textbf{5} \\
    & $\alpha(G)$ & 5 & \# recalculations & 1000 & 1000 & 1000 & 1000 & 881 & 999 & 13 & 512 \\
    \hline
    paley97 & $n$ & 97 & $\hat{\alpha}$ & 8.6 & 7 & 6.3 & 6 & 2 & \textbf{6} & 4 & 0 \\
    & $m$ & 2328 & \# vertices, edges & 13, 22 & 13, 24 & 10, 11 & 9, 6 & 8, 6 & 6, 0 & 4, 0 & 0, 0 \\  
    & $d$ & 0.50 & $\hat{\alpha}_{post}$ & \textbf{6} & \textbf{6} & \textbf{6}& \textbf{6} & \textbf{6} & \textbf{6} & 5 & \textbf{6} \\
    & $\alpha(G)$ & 6 &  \# recalculations & 920 & 999 & 616 & 20 & 777 & 11 & 29 & 362 \\
    \hline
    paley101 & $n$ & 101 & $\hat{\alpha}$ & 7.4 & 7.5 & 4 & 3 & 5 & 0 & 0 & \textbf{5} \\
    & $m$ & 2525 & \# vertices, edges & 16, 40 & 11, 14 & 13, 27 & 11, 16 & 6, 1 & 0, 0 & 0, 0 & 5, 0 \\ 
    & $d$ & 0.50 & $\hat{\alpha}_{post}$ & \textbf{5} & \textbf{5} & \textbf{5} & \textbf{5} & \textbf{5} & \textbf{5} & \textbf{5} & \textbf{5} \\
    & $\alpha(G)$ & 5 & \# recalculations & 1000 & 977 & 1000 & 1000 & 33 & 964 & 454 & 0 \\
    \hline   
    
    \end{tabular}
    \end{adjustbox}
    \end{center}
    \label{table_1}
    \end{table}

    From the results presented in Table~\ref{table_1}, we can see that post-processing did manage to improve the results in $5$ cases when $\beta = \frac{1}{2}$, $8$ cases when $\beta = 1$, $2$ cases when $\beta =  10$ and $6$ cases when $\beta =  100$. When $\beta$ was lower than $\frac{1}{2}$, solutions obtained from QPU had many edges, so post-processing was essential to extract solutions from the samples.
    
    When $\beta = \frac{1}{2}$ is used with post-processing, it outperforms larger $\beta$ values in two cases and matches the performance of larger $\beta$ values in the rest. While this may appear insignificant, in situations where performance equals that of larger $\beta$ values, the optimal value has already been reached, so adjusting the parameter cannot further improve it. There was one exception where the maximum value was not attained, and in this case, $\beta = \frac{1}{2}$ did not outperform higher $\beta$ values. Despite this, we can still confirm that our intuition in $\beta = \frac{1}{2}$ being the most suitable parameter aligns with the outcomes, as it performed equally well, and sometimes even better, than other larger $\beta$ values. Furthermore, we note that parameters $\beta = 10$ and $\beta = 100$ never outperformed $\beta = 1$, and in $6$ instances, they yielded lower solutions.
    
    We observed that decreasing $\beta < \frac{1}{2}$ and utilizing post-processing with $\beta_{post} = \frac{1}{2}$ can produce results equal to or better than those obtained with $\beta=\frac{1}{2}$. Nevertheless, it might be that this enhancement is due to the graph size, as smaller graphs with lower $\beta$ values result in post-processed samples containing more vertices than higher $\beta$ values. For instance, for graphs \verb|C125.9| and \verb|dsjc125.9|, superior solutions are achieved with lower $\beta$ values. When we examine the number of vertices in the solutions being post-processed, such as \verb|C125.9| with $\beta=\frac{1}{10}$, the initial solution contains $62$ vertices, half of all vertices in the graph. Furthermore, the initial solution may not always be the largest post-processed solution; even larger solutions may exist. Considering this, it is reasonable to expect that lower $\beta$ values may perform better for these smaller graphs by allowing larger solutions to be post-processed. 

    A potential drawback of our post-processing method when using $\beta < \frac{1}{2}$ is the increased computational expense compared to $\beta = \frac{1}{2}$ and $\beta > \frac{1}{2}$. This higher cost arises from the need to process a larger quantity of samples. For example, with \verb|C125.9| and $\beta = \frac{1}{10}$, we had to recompute stability numbers on all $1000$ samples, while with $\beta = \frac{1}{2}$, recomputation was needed on only one sample. Additionally, the computational effort varies between $\beta = \frac{1}{2}$ and $\beta > \frac{1}{2}$, with $\beta > \frac{1}{2}$ generally requiring less effort, as seen in instances like \verb|dsjc125.5|. However, this trend reverses for instances such as \verb|paley89| and \verb|paley101|, where $\beta = 1$ required more recalculations than $\beta = \frac{1}{2}$.
    
    Recall that we recalculate the stability numbers of samples by considering their connected components. To explore the structure of the samples and evaluate the complexity of these recalculations, the cardinalities of the largest components considered are presented in Table~\ref{table_1b}. 

    \begin{table}[hbt!]
    \caption{Cardinalities of the largest components considered in the post-processing of results presented in Table~\ref{table_1}}
    \begin{center}
    \begin{adjustbox}{max width=\textwidth}
    \begin{tabular}{ l r | r r r r r r r r }
    \hline
    Instance &  $n$ & $\beta = 1/10$ & $\beta = 1/8$ & $\beta = 1/6$ & $\beta = 1/4$ & $\beta = 1/2$ & $\beta = 1$ & $\beta = 10$ & $\beta = 100$ \\
    \hline
    C125.9 & 125 & 67 & 60 & 47 & 18 & 3 & 1 & 1 & 2 \\
    dsjc125.5 & 125 & 26 & 26 & 19 & 19 & 21 & 14 & 1 & 15 \\
    dsjc125.9 & 125 & 66 & 59 & 45 & 24 & 1 & 1 & 1 & 1 \\
    evil\_chv12x10 & 120 & 71 & 50 & 31 & 12 & 6 & 4 & 1 & 1 \\
    evil\_myc5x24 & 120 & 108 & 95 & 50 & 14 & 4 & 1 & 1 & 1 \\
    evil\_myc11x11 & 121 & 78 & 44 & 23 & 10 & 6 & 4 & 1 & 1 \\
    evil\_s3m25x5 & 125 & 59 & 51 & 45 & 22 & 10 & 4 & 1 & 2 \\
    hamming6\_2 & 64 & 1 & 1 & 1 & 1 & 1 & 1 & 1 & 1 \\
    hamming6\_4 & 64 & 15 & 14 & 16 & 10 & 10 & 11 & 10 & 9 \\
    johnson8\_2\_4 & 28 & 14 & 12 & 10 & 9 & 1 & 1 & 1 & 1 \\
    johnson8\_4\_4 & 70 & 29 & 24 & 1 & 1 & 1 & 1 & 1 & 1 \\
    johnson16\_2\_4 & 120 & 31 & 25 & 8 & 19 & 12 & 8 & 4 & 1 \\
    MANNa\_9 & 45 & 45 & 43 & 13 & 6 & 1 & 1 & 1 & 1 \\
    paley61 & 61 & 20 & 15 & 15 & 12 & 12 & 10 & 8 & 9  \\
    paley73 & 73 & 19 & 20 & 16 & 14 & 11 & 11 & 10 & 9 \\
    paley89 & 89 & 20 & 22 & 22 & 19 & 16 & 16 & 9 & 14  \\
    paley97 & 97 & 18 & 22 & 17 & 12 & 18 & 12 & 11 & 16 \\
    paley101 & 101 & 26 & 18 & 22 & 19 & 12 & 18 & 14 & 1 \\
    \hline
    \end{tabular}
    \end{adjustbox}
    \end{center}
    \label{table_1b}
    \end{table}

    Data presented in Table~\ref{table_1b} indicate that when the calculations are performed with $\beta \in \{\frac{1}{10}, \frac{1}{8}\}$, the components we need to consider are quite large compared to the size of the graph. For example, for the instance \verb|MANN_a9|, the components encompassed nearly all vertices of the graph. The largest components contained $13$ vertices when the calculations were executed with $\beta = \frac{1}{6}$. For all other values of the parameters, we observe that the component sizes were relatively small, suggesting that the recomputation can be carried out efficiently.

    \subsubsection{Results for larger instances}\label{result_large_instances}

    We now investigate how different $\beta$ values, in combination with the post-processing procedure, perform on larger graphs. Since we could not do those computations on QPU, we use the hybrid solver and simulated annealing as initial solvers $S$ in Algorithm~\ref{alg}, while the post-processing is done with simulated annealing, as explained at the beginning of Section~\ref{Different_betas}. Note that although the hybrid solver returns only one solution, we can still post-process it. The results for the hybrid solver are shown in Table~\ref{table_2} and the results for the simulated annealing in Table~\ref{table_3}. These tables contain the same information as Table~\ref{table_1}.

    \begin{table}[hbt!]
    \caption{Numerical results for the post-processing procedure for QUBO formulations   with different values of penalty term $\beta$, where the  hybrid solver is used as the initial solver $S$ in Algorithm~\ref{alg}}
    \scriptsize
    \begin{center}
    \begin{adjustbox}{max width=\textwidth}
    \begin{tabular}{ l l r | l r r r r r r r r }
    \hline
    Instance & Data & Value & Result & $\beta = 1/10$ & $\beta = 1/8$ & $\beta = 1/6$ & $\beta = 1/4$ & $\beta = 1/2$ & $\beta = 1$ & $\beta = 10$ & $\beta = 100$ \\
    \hline
    brock800\_1 & $n$ & 800 & $\hat{\alpha}$ & 26.8 & 25 & 23.3 & 22 & 21 & 21 & 21 & 21  \\
    & $m$ & 112095 & \# vertices, edges & 36, 46 & 32, 28 & 29, 17 & 25, 6 & 22, 1 & 21 0 & 21, 0 & 21, 0   \\ 
    & $d$ & 0.35 & $\hat{\alpha}_{post}$ & 18 & 18 & 18 & 19 & 21 & 21 & 21 & 21  \\
    & $\alpha(G)$ & 23 & \# recalculations & 1 & 1 & 1 & 0 & 0 & 0 & 0 & 0  \\
    \hline
    brock800\_2 & $n$ & 800 & $\hat{\alpha}$ & 26.8 & 25.25 & 23.33 & 22 & 21 & 21 & 21 & 21   \\ 
    & $m$ & 111434 & \# vertices, edges & 35, 41 & 33, 31 & 26, 8 & 23, 2 & 21, 0 & 21, 0 & 21, 0 & 21, 0 \\ 
    & $d$ & 0.35 & $\hat{\alpha}_{post}$ & 20 & 20 & 20 & 21 & 21 & 21 & 21 & 21  \\
    & $\alpha(G)$ & 24 & \# recalculations & 1 & 1 & 1 & 0 & 0 & 0 & 0 & 0  \\
    \hline
    brock800\_3 & $n$ & 800 & $\hat{\alpha}$ & 27 & 25.75 & 24.33 & 21.5 & 22 & 22  & 21 & 21 \\
    & $m$ & 112267 & \# vertices, edges & 33, 30 & 31, 21 & 28, 11 & 24, 5 & 22, 0 & 22, 0 & 21, 0 & 21, 0 \\ 
    & $d$ & 0.35 & $\hat{\alpha}_{post}$ & 22 & 22 & 22 & 20 & 22 & 22 & 21 & 21  \\
    & $\alpha(G)$ & 25 & \# recalculations & 1 & 1 & 1 & 1 & 0 & 0 & 0 & 0  \\
    \hline
    brock800\_4 & $n$ & 800 & $\hat{\alpha}$ & 26.6 & 24.75 & 23 & 22.5 & 21 & 20  & 21 & 21 \\
    & $m$ & 111957 & \# vertices, edges & 36, 47 & 32, 29 & 27, 12 & 25, 5 & 22, 1 & 20, 0  & 21, 0 & 21, 0 \\  
    & $d$ & 0.35 & $\hat{\alpha}_{post}$ & 18 & 18 & 19 & 20 & 21 & 20 & 21 & 21  \\
    & $\alpha(G)$ & 26 & \# recalculations & 1 & 1 & 1 & 0 & 0 & 0 & 0 & 0  \\
    \hline
    p\_hat1500\_1 & $n$ & 1500 & $\hat{\alpha}$ & 14.8 & 13.75 & 13 & 12.5 & \textbf{12} & 11 & 11 & 11   \\
    & $m$ & 839327 & \# vertices, edges & 20, 26 & 18, 17 & 15, 6 & 14, 3 & 12, 0 & 11, 0 & 11, 0 & 11, 0 \\ 
    & $d$ & 0.74 & $\hat{\alpha}_{post}$ & 10 & 10 & 11 & 11 & \textbf{12} & 11  & 11 & 11 \\
    & $\alpha(G)$ & 12 & \# recalculations & 1 & 1 & 1 & 0 & 0 & 0 & 0 & 0  \\
    \hline
    p\_hat1500\_2 & $n$ & 1500 & $\hat{\alpha}$ & 86.2 & 80.5 & 74.67 & 68.5 & 65 & \textbf{65}  & \textbf{65} & \textbf{65} \\
    & $m$ & 555290 & \# vertices, edges & 113, 134 & 105, 98 & 91, 49 & 78, 19 & 68, 3 & 65, 0 & 65, 0 & 65, 0 \\ 
    & $d$ & 0.49 & $\hat{\alpha}_{post}$ & 63 & 62 & 63 & 63 & \textbf{65} & \textbf{65} & \textbf{65} & \textbf{65}   \\
    & $\alpha(G)$ & 65 & \# recalculations & 1 & 1 & 1 & 1 & 0 & 0 & 0 & 0 \\
    \hline
    p\_hat1500\_3 & $n$ & 1500 & $\hat{\alpha}$ & 122.4 & 114.75 & 107 & 99.5 & 94 & \textbf{94} & \textbf{94} & \textbf{94}   \\
    & $m$ & 277006 & \# vertices, edges & 162, 198 & 146, 125 & 128, 63 & 112, 25 & 97, 3 & 94, 0 & 94, 0 & 94, 0  \\
    & $d$ & 0.25 & $\hat{\alpha}_{post}$ & 89 & 89 & 88 & 92 & \textbf{94} & \textbf{94}  & \textbf{94} & \textbf{94} \\
    & $\alpha(G)$ & 94 & \# recalculations & 1 & 1 & 1 & 1 & 0 & 0 & 0 & 0  \\
    \hline
    
    \end{tabular}
    \end{adjustbox}
    \end{center}
    \label{table_2}
    \end{table}

    \begin{table}[hbt!]
    \caption{Numerical results for the post-processing procedure for QUBO formulations with different values of penalty term $\beta$, where the simulated annealing is used as the initial solver $S$ in Algorithm~\ref{alg}}
    \scriptsize
    \begin{center}
    \begin{adjustbox}{max width=\textwidth}
    \begin{tabular}{ l l r | l r r r r r r r r }
    \hline
    Instance & Information & Value & Result & $\beta = 1/10$ & $\beta = 1/8$ & $\beta = 1/6$ & $\beta = 1/4$ & $\beta = 1/2$ & $\beta = 1$ & $\beta = 10$ & $\beta = 100$ \\
    \hline
    brock800\_1 & $n$ & 800 & $\hat{\alpha}$ & 26.8 & 25 & 23 & 21.5 & 21 & 20  & 17 & 15  \\
    & $m$ & 112095 & \# vertices, edges & 35, 41 & 31, 24 & 25, 6 & 25, 7 & 21, 0 & 20, 0 & 17, 0 & 15, 0 \\  
    & $d$ & 0.35 & $\hat{\alpha}_{post}$ & 21 & 21 & 21 & 21 & 21 & 20  & 17 & 15  \\
    & $\alpha(G)$ & 23 & \# recalculations & 919 & 318 & 69 & 2 & 0 & 0 & 0 & 0 \\
    \hline
    brock800\_2 & $n$ & 800 & $\hat{\alpha}$ & 26.8 & 25.25 & 23.33 & 22 & 21 & 20 & 17 & 15  \\
    & $m$ & 111434 & \# vertices, edges & 33, 31 & 33, 31 & 27, 11 & 24, 4 & 22, 1 & 20, 0 & 17, 0 & 15, 0 \\     & $d$ & 0.35 & $\hat{\alpha}_{post}$ & 20 & 21 & 21 & 21 & 21 & 20 & 17 & 15  \\
    & $\alpha(G)$ & 24 & \# recalculations & 995 & 349 & 68 & 1 & 0 & 0 & 0 & 0  \\
    \hline
    brock800\_3 & $n$ & 800 & $\hat{\alpha}$ & 27 & 25.75 & 24.33 & 22.5 & 22 & 22 & 17 & 15  \\
    & $m$ & 112267 & \# vertices, edges & 32, 25 & 32, 25 & 28, 11 & 25, 5 & 23, 1 & 22, 0 & 17, 0 & 15, 0 \\ 
    & $d$ & 0.35 & $\hat{\alpha}_{post}$ & 22 & 22 & 22 & 22 & 22 & 22 & 17 & 15  \\
    & $\alpha(G)$ & 25 & \# recalculations & 465 & 29 & 4 & 1 & 0 & 0 & 0 & 0 \\
    \hline
    brock800\_4 & $n$ & 800 & $\hat{\alpha}$ & 26.4 & 24.75 & 23.3 & 22 & 21 & 20  & 17 & 15 \\
    & $m$ & 111957 & \# vertices, edges & 34, 38 & 32, 29 & 25, 5 & 23, 2 & 21, 0 & 20, 21 & 17, 0 & 15, 0 \\ 
    & $d$ & 0.35 & $\hat{\alpha}_{post}$ & 21 & 21 & 21 & 21 & 21 & 20 & 17 & 15   \\
    & $\alpha(G)$ & 26 & \# recalculations & 870 & 362 & 58 & 0& 0 & 0 & 0 & 0  \\
    \hline
    p\_hat1500\_1 & $n$ & 1500 & $\hat{\alpha}$ & 14.8 & 13.75 & 13 & 12.5 & \textbf{12} & 11 & 9 & 9  \\
    & $m$ & 839327 & \# vertices, edges & 21, 31 & 19, 21 & 16, 9 & 14, 3 & 12, 0 & 11, 0 & 9, 0 & 9, 0  \\ 
    & $d$ & 0.74 & $\hat{\alpha}_{post}$ & 11 & \textbf{12} & 11 & \textbf{12} & \textbf{12} & 11 & 9 & 9  \\
    & $\alpha(G)$ & 12 & \# recalculations & 954 & 480 & 98 & 1 & 0 & 0 & 0 & 0   \\
    \hline
    p\_hat1500\_2 & $n$ & 1500 & $\hat{\alpha}$ & 86.2 & 80.5 & 74.67 & 68.5 & 65 & \textbf{65} & 56 & 43   \\
    & $m$ & 555290 & \# vertices, edges & 116, 149 & 104, 94 & 90, 46 & 76, 15 & 67, 2 & 65, 0 & 56, 0 & 43, 0 \\  
    & $d$ & 0.49 & $\hat{\alpha}_{post}$ & 63 & 63 & 63 & 64 & \textbf{65} & \textbf{65} & 56 & 43   \\
    & $\alpha(G)$ & 65 & \# recalculations & 1000 & 1000 & 1000 & 152 & 0 & 0 & 0 & 0   \\
    \hline
    p\_hat1500\_3 & $n$ & 1500 & $\hat{\alpha}$ & 122.4 & 114.75 & 107 & 99.5 & 94 & \textbf{94} & 76 & 59  \\
    & $m$ & 277006 & \# vertices, edges & 158, 178 & 144, 117 & 134, 81 & 110, 21 & 99, 5 & 94, 0  & 76, 0 & 59. 0 \\ 
    & $d$ & 0.25 & $\hat{\alpha}_{post}$ & 91 & 91 & 91 & 93 & \textbf{94} & \textbf{94}  & 76 & 59 \\
    & $\alpha(G)$ & 94 & \# recalculations & 1000 & 1000 & 1000 & 369 & 0 & 0 & 0 & 0   \\
    \hline
    \end{tabular}
    \end{adjustbox}
    \end{center}
    \label{table_3}
    \end{table}

    From the results presented in Table~\ref{table_2}, we observe that using lower $\beta$ values with the hybrid solver generally led to worse results compared to $\beta=\frac{1}{2}$, with an exception being the graph \verb|brock800_2| when $\beta=\frac{1}{4}$. The results for simulated annealing shown in Table~\ref{table_3} are slightly better for lower $\beta$ values, but still, in none of the cases did lower $\beta$ values yield superior results, even though the optimal value has not been reached yet. Interestingly, in graphs \verb|p_hat1500_3| and \verb|p_hat1500_2| not a single lower $\beta$ has matched the result of $\beta=\frac{1}{2}$. This again supports the interpretation that the reason why lower values of $\beta$  performed better than larger $\beta$ values on instances shown in Table~\ref{table_1} is mainly the graph size (in   Table~\ref{table_1} the instances are smaller, so we can afford re-computations of many relatively larger samples), and not because the lower $\beta$ values are a more appropriate choice for this problem.

    When we compare results for $\beta = \frac{1}{2}$ and $\beta = 1$, we note that $\beta = \frac{1}{2}$ gave better results for $2$ instances when the hybrid solver was used, and $4$ instances when the simulated annealing was employed. Results for $\beta = 10$ and $\beta = 100$ show that the hybrid solver still performs well, while simulated annealing demonstrates much worse performance, i.e., the choice of penalty term in the QUBO formulation significantly affects the performance of simulated annealing. 

    Regarding the complexity of the post-processing procedure, when using the hybrid solver, only one solution is obtained, making post-processing a less intensive task. For computations done with simulated annealing, we analyze the sizes of the largest connected components of post-processed samples and present this information in Table~\ref{table_3b}.

    \begin{table}[hbt!]
    \caption{Cardinalities of the largest components considered in the post-processing of results presented in Table~\ref{table_3}}
    \begin{center}
    \begin{adjustbox}{max width=\textwidth}
    \begin{tabular}{ l l | r r r r r r r r }
    \hline
    Instance & $n$  & $\beta = 1/10$ & $\beta = 1/8$ & $\beta = 1/6$ & $\beta = 1/4$ & $\beta = 1/2$ & $\beta = 1$ & $\beta = 10$ & $\beta = 100$ \\
    \hline
    brock800\_1 & 800 & 38 & 33 & 18 & 3 & 1 & 1 & 1 & 1 \\
    brock800\_2 & 800 & 38 & 35 & 16 & 2 & 1 & 1 & 1 & 1 \\
    brock800\_3 & 800 & 38 & 31 & 8 & 5 & 1 & 1 & 1 & 1 \\
    brock800\_4 & 800 & 37 & 32 & 17 & 1 & 1 & 1 & 1 & 1 \\
    p\_hat1500\_1 & 500 & 22 & 20 & 14 & 3 & 1 & 1 & 1 & 1 \\ 
    p\_hat1500\_2 & 1500 & 112 & 95 & 34 & 11 & 1 & 1 & 1 & 1 \\
    p\_hat1500\_3 & 1500 & 156 & 128 & 34 & 7 & 1 & 1 & 1 & 1 \\
    \hline
    \end{tabular}
    \end{adjustbox}
    \end{center}
    \label{table_3b}
    \end{table}

    From Table~$\ref{table_3b}$, we can see that the sizes of connected components are larger when lower values of $\beta$ are used. However, for $\beta \in \{\frac{1}{6}, \frac{1}{4}\}$, the sizes of components are relatively small. This observation has also been seen in smaller instances, as shown in Table~$\ref{table_1b}$. It is interesting to note this pattern considering that the graphs now are significantly larger compared to those in Table~$\ref{table_1b}$. Despite working with graphs with over $800$ vertices, the largest components we recalculated stability numbers for had a maximum of $34$ vertices. Finally, we note that for $\beta \geq \frac{1}{2}$, there were no instances where improved solutions could be obtained, leading us to skip the recalculation of the stability numbers for all returned samples.

    \subsection{Solving larger instances with partitioning}\label{Larger_instances_partitioning_results}

    As the final set of experiments, we consider large graphs that cannot be embedded on QPU and employ the method of the simple CH-partitioning as explained in Section~\ref{Graph_partitioning} and outlined in Algorithm~\ref{alg_part}. Hence, we first partition the graph into subproblems that can be embedded on QPU, run $100$ iterations on QPU, and then post-process the samples with simulated annealing, where the number of runs is also set to $100$. Since our previous computational investigation suggests that the best choice for the penalty parameter is $\frac{1}{2}$, we set $\beta = \beta_{post} = \frac{1}{2}$. The vertices are ordered such that $d(v_1) \leq d(v_2) \leq \cdots \leq d(v_n)$, allowing us to reduce the size of partitions. To evaluate the results, we compare them with those obtained using the hybrid solver.
    
    The results are presented in Table~\ref{table_4}. Columns $1-5$ contain the basic information about the considered instances: name of the instance, number of vertices $n$, number of edges $m$, edge density $d$ and the stability number $\alpha (G)$. In columns $6-9$, the results for the hybrid solver and the simple partitioning are presented. The results for the hybrid solver contain information about the solution that was post-processed and the number of vertices and edges in the solution before it was post-processed. The simple partitioning results contain the best result from all post-processed partitions and the number of partitions we considered. We write the value in bold if the optimum was achieved by either the hybrid solver or the simple partitioning method.

    \begin{table}[hbt!]
    \caption{Results for the computations with simple graph partitioning method combined with the post-processing procedure}
    \scriptsize
    \begin{center}
    \begin{adjustbox}{max width=\textwidth}
    \begin{tabular}{ l r r r r | l r | l r }
    \hline
    Instance & $n$ & $m$ & $d$ & $\alpha(G)$ &  \multicolumn{2}{c}{Hybrid solver} &  \multicolumn{2}{c}{Simple CH-partitioning}  \\
    \hline
    keller4 & 171 & 5100 & 0.35 & 11 & $\hat{\alpha}_{post}$ & \textbf{11} & $\hat{\alpha}_{post}^{part}$ & \textbf{11} \\
    & & & & & \# vertices, edges &  13, 2 & \# partitions & 147/171 \\
    \hline 
    brock200\_1 & 200 & 5066 & 0.25 & 21 & $\hat{\alpha}_{post}$ & \textbf{21} & $\hat{\alpha}_{post}^{part}$ & 19 \\
    & & & & & \# vertices, edges &  21, 0 & \# partitions & 161/200 \\
    \hline
    brock200\_2 & 200 & 10024 & 0.50 & 12 & $\hat{\alpha}_{post}$ & \textbf{12} & $\hat{\alpha}_{post}^{part}$ & 10 \\
    & & & & & \# vertices, edges &  12, 0 & \# partitions & 170/200 \\
    \hline 
    brock200\_3 & 200 & 7852 & 0.39 & 15 & $\hat{\alpha}_{post}$ & 14 & $\hat{\alpha}_{post}^{part}$ & 13 \\
    & & & & & \# vertices, edges &  14, 0 & \# partitions & 166/200 \\
    \hline 
    brock200\_4 & 200 & 6811 & 0.34 & 17 & $\hat{\alpha}_{post}$ & 16 & $\hat{\alpha}_{post}^{part}$ & 14 \\
    & & & & & \# vertices, edges &  16, 0 & \# partitions & 170/200 \\
    \hline 
    san200\_0\_7\_1 & 200 & 5970 & 0.30 & 30 & $\hat{\alpha}_{post}$ & 16 & $\hat{\alpha}_{post}^{part}$ & 16 \\
    & & & & & \# vertices, edges &  24, 8 & \# partitions & 174/200 \\
    \hline 
    san200\_0\_7\_2 & 200 & 5970 & 0.30 & 18 & $\hat{\alpha}_{post}$ & 13 & $\hat{\alpha}_{post}^{part}$ & 13 \\
    & & & & & \# vertices, edges &  20, 7 & \# partitions & 182/200 \\
    \hline 
    sanr200\_0\_7 & 200 & 6032 & 0.30 & 18 & $\hat{\alpha}_{post}$ & \textbf{18} & $\hat{\alpha}_{post}^{part}$ & 17 \\
    & & & & & \# vertices, edges &  18, 0 & \# partitions & 166/200 \\
    \hline
    c-fat200-1 & 200 & 18366 & 0.92 & 12 & $\hat{\alpha}_{post}$ & \textbf{12} & $\hat{\alpha}_{post}^{part}$ & \textbf{12} \\
    & & & & & \# vertices, edges &  12, 0 & \# partitions & 3/200 \\
    \hline
    c-fat200-2 & 200 & 16665 & 0.84 & 24 & $\hat{\alpha}_{post}$ & \textbf{24} & $\hat{\alpha}_{post}^{part}$ & \textbf{24} \\
    & & & & & \# vertices, edges &  24, 0 & \# partitions & 3/200 \\
    \hline
    c-fat200-5 & 200 & 11427 & 0.57 & 58 & $\hat{\alpha}_{post}$ & \textbf{58} & $\hat{\alpha}_{post}^{part}$ & \textbf{58} \\
    & & & & & \# vertices, edges &  58, 0 & \# partitions & 3/200 \\
    \hline
    c-fat500-1 & 500 & 120291 & 0.96 & 14 & $\hat{\alpha}_{post}$ & \textbf{14} & $\hat{\alpha}_{post}^{part}$ & \textbf{14} \\
    & & & & & \# vertices, edges &  58, 0 & \# partitions & 3/500 \\
    \hline
    c-fat500-2 & 500 & 115611 & 0.93 & 26 & $\hat{\alpha}_{post}$ & \textbf{26} & $\hat{\alpha}_{post}^{part}$ & \textbf{26} \\
    & & & & & \# vertices, edges &  26, 0 & \# partitions & 3/500 \\
    \hline
    c-fat500-5 & 500 & 101559 & 0.81 & 64 & $\hat{\alpha}_{post}$ & \textbf{64} & $\hat{\alpha}_{post}^{part}$ & \textbf{64} \\
    & & & & & \# vertices, edges &  64, 0 & \# partitions & 3/500 \\
    \hline
    p\_hat500\_1 & 500 & 93181 & 0.75 & 9 & $\hat{\alpha}_{post}$ & \textbf{9} & $\hat{\alpha}_{post}^{part}$ & \textbf{9} \\
    & & & & & \# vertices, edges &  9, 0 & \# partitions & 470/500 \\
    \hline 
    \end{tabular}
    \end{adjustbox}
    \end{center}
    \label{table_4}
    \end{table}

    From the results presented in Table~\ref{table_4}, we note that for instances with densities up to $0.50$, the hybrid solver performed better than the simple partitioning. More precisely, the hybrid solver gave better solutions for $5$ of $8$ instances. However, the optimal value was not always attained. In contrast, both methods yielded optimal results for $7$ instances with densities higher than $0.50$. 
    
    We note that instances \verb|san200_0_7_1| and \verb|san200_0_7_2| appear to be the most challenging for both methods. Despite producing identical results, the deviation from the optimal values for these instances is substantial. Other challenging instances are \verb|brock| graphs. For $4$ considered instances, the hybrid solver found $2$ optimal solutions, while the partitioning method yielded suboptimal results.
    
    Nevertheless, the simple partitioning method performed extremely well on \verb|c-fat| instances. The proposed bounding procedure significantly reduced the number of partitions, and we obtained optimal solutions for all $6$ instances by considering only $3$ partitions for each instance.
    
    Altogether, the presented initial results demonstrate the potential of the proposed simple partitioning. While the hybrid solver gave better solutions for certain instances, the differences in results were no more than $2$ vertices. Overall, by employing a few simple techniques, we attained results similar to those obtained by the hybrid solver.

	\section{Conclusion}\label{Conclusion}
	
	This paper presents a detailed analysis of solutions for the stable set problem generated by the D-Wave's quantum annealer, providing both theoretical insights and computational findings. We formulated the stable set problem as a QUBO problem with quadratic constraints weighted by a penalty parameter $\beta$. We showed that with an exact QUBO solver, the optimal value for $\beta\ge\frac{1}{2}$ is the stability number of the graph. However, due to the heuristic nature of quantum annealing, solutions from the D-Wave's QPU are often far from optimal values and may not even represent stable sets, as indicated by our computational results for instances with up to $125$ vertices. To address these challenges, we proposed a post-processing procedure that enables us to detect samples that could lead to improved solutions and to extract solutions that are stable sets. Also, we provided theoretical guarantees about the quality of the extracted solutions.

    We have employed the post-processing procedure on graphs that can be embedded on QPU but also on larger instances, for which we used the D-Wave hybrid solver and the classical simulated annealing solver from the D-Wave’s library for interacting with QPU. A comparison of the results for large instances revealed that the solutions obtained using the hybrid solver and simulated annealing are very similar. This may partially explain how the hybrid solver is running.  
    
     Computational results showed that the post-processing procedure greatly improved solution quality. Currently, selected samples are post-processed using simulated annealing, but alternative approaches could further enhance outcomes. Due to the small and sparse nature of the post-processed samples, an exact method may yield superior results in a shorter time.

    Since not all instances can be embedded on the D-Wave's QPU, we explored methods to address this issue. We used the simple CH-partitioning method as an alternative to the D-Wave's hybrid solver for certain instances. In this approach, we break down the original problem into smaller subproblems, solve them using the QPU, and then take the largest stable set across the subproblems as the final solution. To improve the quality of results, we combined the proposed partitioning method with the post-processing procedure and obtained promising results. Future research should concentrate on creating techniques to break down partitions further, recognizing partitions that could lead to optimal solutions, and investigating additional efficient reduction strategies to decrease computational costs. In particular, it would be valuable to investigate the impact of different vertex orderings.

    Furthermore, in order to assess the impact of different $\beta$ values on the quality of solutions, we conducted thorough numerical experiments with various values of the penalty term $\beta$ ranging from $\frac{1}{10}$ to $100$. Our computational results indicated that $\beta = \frac{1}{2}$ produced the best quality solutions, which is in accordance with the given theoretical guarantees. Nevertheless, it would be valuable to conduct further computational experiments by exploring values of $\beta$ in the form $\beta = \frac{1}{2} \pm \epsilon$, where $\epsilon$ is small.

    Finally, the only hyperparameter we adjusted in the D-Wave solvers for the computations presented in this work was the number of runs, while all other hyperparameters were left at their default settings. Future research should explore varying additional parameters of the QPU, such as annealing time, spin reversal transforms, and chain strength. Optimizing these parameters could enhance the performance of the QPU before applying the post-processing procedure.
	
	\section*{Disclosure statement}
	
	The authors report there are no competing interests to declare.

    \section*{Acknowledgments} We thank the two anonymous reviewers for their insightful comments and suggestions, which have substantially improved this paper.
	
	\section*{Data availability statement}
	
	The program code associated with this paper is available as ancillary files from the arXiv page of this paper: arXiv:2405.12845. Additionally, the source code and data are available upon request from the authors.
	
	\section*{Funding} The research of the first and second authors was funded by the Slovenian Research and Innovation Agency (ARIS) through the annual work program of Rudolfovo. The second author was also partially funded by research program P2-0162. The research of the third author was funded by the Austrian Science Fund (FWF) [10.55776/DOC78]. For the purposes of open access, the authors have applied a CC BY public copyright license to all author-accepted manuscript versions resulting from this submission.

	\afterpage{\clearpage}
	
	\bibliographystyle{plain}
	\bibliography{References}
	
\end{document}